\documentclass[11pt]{amsart}
\usepackage{amsmath}
\usepackage{amssymb}
\usepackage{amsthm}
\usepackage{latexsym}
\usepackage{graphicx}
\usepackage{hyperref}
\usepackage{enumerate}
\usepackage[all]{xy}

\newtheorem{thm}{Theorem}[section]
\newtheorem{cor}[thm]{Corollary}
\newtheorem{lem}[thm]{Lemma}
\newtheorem{prop}[thm]{Proposition}

\newtheorem{thmintro}{Theorem}

\newtheorem{propintro}[thmintro]{Proposition}

\newcommand{\Z}{\mathbb Z}
\newcommand{\Q}{\mathbb Q}
\newcommand{\R}{\mathbb R}
\newcommand{\C}{\mathbb C}
\newcommand{\mf}{\mathfrak}
\newcommand{\mc}{\mathcal}
\newcommand{\mb}{\mathbf}
\newcommand{\mh}{\mathbb}
\def\Irr{{\rm Irr}}
\newcommand{\mr}{\mathrm}
\newcommand{\enuma}[1]{\begin{enumerate}[\textup{(}a\textup{)}] {#1} \end{enumerate}}
\newcommand{\Fr}{\mathrm{Frob}}
\newcommand{\Sc}{\mathrm{sc}}
\newcommand{\ad}{\mathrm{ad}}
\newcommand{\cusp}{\mathrm{cusp}}
\newcommand{\nr}{\mathrm{nr}}
\newcommand{\Wr}{\mathrm{wr}}
\newcommand{\cpt}{\mathrm{cpt}}
\newcommand{\Rep}{\mathrm{Rep}}
\newcommand{\af}{\mathrm{aff}}
\def\tor{{\rm tor}}
\newcommand{\unip}{\mathrm{unip}}
\newcommand{\der}{\mathrm{der}}
\newcommand{\matje}[4]{\left(\begin{smallmatrix} #1 & #2 \\ 
#3 & #4 \end{smallmatrix}\right)}

\begin{document}

\title{A local Langlands correspondence \\ for unipotent representations}
\date{\today}
\thanks{The author is supported by a NWO Vidi grant "A Hecke algebra approach to the 
local Langlands correspondence" (nr. 639.032.528).}
\subjclass[2010]{Primary 22E50; Secondary 11S37, 20G25}
\maketitle

\ 
\begin{center}
{\Large Maarten Solleveld} \\[1mm]
IMAPP, Radboud Universiteit \\
Heyendaalseweg 135, 6525AJ Nijmegen, the Netherlands \\
email: m.solleveld@science.ru.nl \\[7mm] \
\end{center}

\begin{abstract}
Let $G$ be a connected reductive group over a non-archimedean local field $K$,
and assume that $G$ splits over an unramified extension of $K$. We establish
a local Langlands correspondence for irreducible unipotent representations of $G$.
It comes as a bijection between the set of such representations and the collection
of enhanced L-parameters for $G$, which are trivial on the inertia subgroup of
the Weil group of $K$. We show that this correspondence has many of the expected 
properties, for instance with respect to central characters, tempered representations,
the discrete series, cuspidality and parabolic induction.

The core of our strategy is the investigation of affine Hecke algebras on both sides of 
the local Langlands correspondence. When a Bernstein component of $G$-representations 
is matched with a Bernstein component of enhanced L-parameters, we prove a comparison 
theorem for the two associated affine Hecke algebras.

This generalizes work of Lusztig in the case of adjoint $K$-groups.\\[8mm]
\end{abstract}

\tableofcontents

\newpage

\section*{Introduction}

Let $K$ be a non-archimedean local field and let $\mc G$ be a connected reductive $K$-group.
We consider smooth, complex representations of the group $G = \mc G (K)$. An irreducible 
smooth $G$-representation $\pi$ is called unipotent if there exists a parahoric subgroup
$P_{\mf f} \subset G$ and an irreducible $P_{\mf f}$-representation $\sigma$, which is
inflated from a cuspidal representation of the finite reductive quotient of $P_{\mf f}$,
such that $\pi |_{P_{\mf f}}$ contains $\sigma$. These notions behave best when $\mc G$
splits over an unramified extension of $K$, so that assume that in the introduction (and
in most of the paper).

We will exhibit a local Langlands correspondence (LLC) for all irreducible unipotent
representations of such reductive $p$-adic groups. This generalizes results of Lusztig
\cite{LusUni1,LusUni2} for simple adjoint $K$-groups.

Let us make the statement more precise, referring to Section \ref{sec:2} for the details.
We denote the set of isomorphism classes of irreducible unipotent $G$-representations by 
$\Irr_\unip (G)$. As usual, we consider Langlands parameters 
\[
\phi : \mb W_K \times SL_2 (\C) \longrightarrow {}^L G  = G^\vee \rtimes \mb W_K .
\]
As component group of we take the group $\mc S_\phi$ from \cite{Art,HiSa,ABPSLLC}.
An enhancement of $\phi$ is an irreducible representation $\rho$ of $\mc S_\phi$,
and there is a $G$-relevance condition for such enhancements. We let $\Phi_e (G)$
be the set of $G^\vee$-association classes of $G$-relevant enhanced L-parameters
$\mb W_K \times SL_2 (\C) \to {}^L G$.

Let $\mb I_K$ be the inertia subgroup of the Weil group $\mb W_K$. An enhanced
L-parameter $(\phi,\rho)$ is called unramified if $\phi (w) = (1,w)$ for all
$w \in \mb I_K$. We denote the resulting subset of $\Phi_e (G)$ by $\Phi_{\nr,e}(G)$.

\begin{thmintro}\label{thm:A} \textup{(see Section \ref{sec:LLC})} \\
There exists a bijection
\[
\begin{array}{ccc}
\Irr_\unip (G) & \longrightarrow & \Phi_{\nr,e}(G) \\
\pi & \mapsto & (\phi_\pi, \rho_\pi) \\
\pi (\phi,\rho) & \text{\rotatebox[origin=c]{180}{$\mapsto$}} & (\phi,\rho)
\end{array}
\]
with the following properties.
\begin{enumerate}[(a)]
\item Compatibility with direct products of reductive $K$-groups.
\item Equivariance with respect to the canonical actions of the group $X_\Wr (G)$ of
weakly unramified characters of $G$.
\item The central character of $\pi$ equals the character of $Z(G)$ determined by $\phi_\pi$.
\item $\pi$ is tempered if and only if $\phi_\pi$ is bounded.
\item $\pi$ is essentially square-integrable if and only if $\phi_\pi$ is discrete.
\item $\pi$ is supercuspidal if and only if $(\phi_\pi, \rho_\pi)$ is cuspidal.
\item The local Langlands correspondences for the Levi subgroups of $G$ and the cuspidal
support maps form a commutative diagram
\[
\begin{array}{ccc}
\Irr_\unip (G) & \longrightarrow & \Phi_{\nr,e}(G) \\
\downarrow & & \downarrow \\
\bigsqcup_L \Irr_{\cusp,\unip}(L) \big/ (N_G (L)/L)  & 
\longrightarrow & \bigsqcup_L \Phi_{\nr,\cusp}(L) \big/ (N_G (L) / L) 
\end{array} .
\]
Here $L$ runs over a collection of representatives for the conjugacy classes of Levi subgroups 
of $G$. See Section \ref{sec:2} for explanation of the notation in the diagram.
\item Suppose that $P = L U$ is a parabolic subgroup of $G$ and that 
$(\phi,\rho^L) \in \Phi_{\nr,e}(L)$ is bounded. Then the normalized parabolically inducted 
representation $I_P^G \pi (\phi,\rho^L)$ is a direct sum of representations $\pi (\phi,\rho)$, \
with multiplicities $[\rho^L : \rho]_{\mc S_{\phi}^L}$.
\item Compatibility with the Langlands classification for representations of reductive groups
and the Langlands classification for enhanced L-parameters.
\end{enumerate}
\end{thmintro}

Since there are so many properties, one may wonder to what extent the LLC is characterized by them.
First we note that $\phi_\pi$ is certainly not uniquely determined by $\pi$ alone. Namely,
in many cases one can twist the LLC by a character of $\mc S_{\phi_\pi}$ and retain all the above
properties. 

The obvious next question is: do the above conditions determine the map 
\[
\Irr_\unip (G) \to \Phi_\nr (G) : \pi \mapsto \phi_\pi
\]
uniquely? Again the answer is no, for (sometimes) one can still adjust the map by the action of
a weakly unramified character of $G$. 

Then one may enquire whether $\pi \mapsto \phi_\pi$ is canonical up to twists by elements of
$X_\Wr (G)$. That is the case, and there are two ways to see it.
\begin{itemize}
\item With formal degrees \cite{FOS,Opd18}. Namely, 
\[
\Irr_\unip (G) \to \Phi_\nr (G) / X_\Wr (G) : \pi \mapsto X_\Wr (G) \phi_\pi 
\]
is the unique map which makes the Hiraga--Ichino--Ikeda conjectures \cite{HII} true in the 
supercuspidal case \cite[\S 16]{FOS} and makes the HII conjectures ``almost true" in general
\cite[Theorem 3.5.1]{Opd18}.
\item With functoriality and Lusztig's work \cite{LusUni1,LusUni2}. 
Although Lusztig does not make it explicit, he indicates in \cite[\S 6.6]{LusUni1} that his 
LLC is canonical (up to weakly unramified twists) when $\mc G$ is adjoint and simple. 
Our general case is related to the adjoint (simple) case by functoriality, 
which implies that $\pi \mapsto \phi_\pi$ is essentially unique.
\end{itemize}
\vspace{3mm}
Now we provide an overview of the setup and the general strategy of the paper. Foremostly,
everything runs via affine Hecke algebras. Commonly an affine Hecke algebra is associated to
one Bernstein component in $\Irr (G)$. To get them into play on the Galois side of the LLC,
one first needs a good notion of a Bernstein component there. That was achieved in \cite{AMS1},
by means of a cuspidal support map for enhanced L-parameters. (To this end the enhancements are 
essential. without them one cannot even define cuspidality of L-parameters.) In \cite{AMS1} the
cuspidal support of an enhanced L-parameter for $G$ is given as the $G^\vee$-association class
of a cuspidal L-parameter for a $G$-relevant Levi subgroup of ${}^L G$. For later comparison,
we need to translate this to a cuspidal L-parameter for a Levi subgroup of $G$, unique up to
$G$-conjugation. That is the purpose of the next result (which we actually prove in greater
generality).

\begin{propintro}\label{prop:B} \textup{(see Corollary \ref{cor:1.3})} \\
There exists a canonical bijection between:
\begin{itemize}
\item the set of $G$-conjugacy classes of Levi subgroups of $G$;
\item the set of $G^\vee$-conjugacy classes of $G$-relevant Levi
subgroups of ${}^L G$.
\end{itemize}
\end{propintro}

In Section \ref{sec:2} we show how one can associate, to every Bernstein component
$\Phi_e (G)^{\mf s^\vee}$ of enhanced L-parameters for $G$, an affine Hecke algebra 
$\mc H (\mf s^\vee,\vec{v})$. Here the array of complex parameters $\vec{v}$ can be chosen
freely. This relies entirely on \cite{AMS3}. The crucial properties of this algebra are:
\begin{itemize}
\item the irreducible representations of $\mc H (\mf s^\vee,\vec{v})$ are canonically
parametrized by $\Phi_e (G)^{\mf s^\vee}$ (at least when the parameters are chosen in
$\R_{>0}$);
\item the maximal commutative subalgebra of $\mc H (\mf s^\vee,\vec{v})$ (coming from the
Bernstein presentation) is the ring of regular functions on the complex torus $\mf s_L^\vee$
(which forms a cuspidal Bernstein component of enhanced L-parameters for a Levi subgroup
$L$ of $G$).
\end{itemize}
Only after that we really turn to unipotent $G$-representations. From work of Morris and
Lusztig \cite{Mor1,Mor2,LusUni1} it is known that every Bernstein block
Rep$(G)_{\mf s}$ of smooth unipotent $G$-representations admits a type, and that it is 
equivalent to the module category of an affine Hecke algebra. In the introduction, we will
denote that algebra simply by $\mc H_{\mf s}$. 
In Section \ref{sec:3} we work out the Bernstein presentation of $\mc H_{\mf s}$, that is,
we make the underlying torus and Weyl group explicit in terms of $\mf s$.

Armed with a good understanding of the affine Hecke algebras on both sides of the LLC,
we set out to compare them. Here we make good use of a local Langlands correspondence for
supercuspidal unipotent representations, which was established in \cite{FOS}.
Together with Proposition \ref{prop:B} that gives rise to:

\begin{propintro}\label{prop:C}
There exists a bijection, canonical up to twists by weakly unramified characters, between:
\begin{itemize}
\item the set $\mf{Be} (G)_\unip$ of Bernstein components in $\Irr (G)$ consisting of
unipotent representations;
\item the set $\mf{Be}^\vee (G)_\nr$ of Bernstein components in $\Phi_e (G)$ consisting
of unramified enhanced L-parameters.
\end{itemize}
\end{propintro}

When $\mf s \in \mf{Be}(G)_\unip$ corresponds to $\mf s^\vee \in \mf{Be}^\vee (G)_\nr$ 
via Proposition \ref{prop:C}, we show that the two associated Hecke algebras
$\mc H_{\mf s}$ and $\mc H (\mf s^\vee, \vec{v})$ have isomorphic Weyl groups, and that
the underlying tori are isomorphic (via the LLC on the cuspidal level). By reduction to
the case of adjoint groups, which was settled in \cite{LusUni1,LusUni2}, we prove that
the labels of these two affine Hecke algebras match. That leads to:

\begin{thmintro}\label{thm:D}
When $\mf s$ corresponds to $\mf s^\vee$ via Proposition \ref{prop:C}, 
$\mc H (\mf s^\vee,\vec{v})$ is canonically isomorphic to $\mc H_{\mf s}$, 
for an explicit choice of the parameters $\vec v$.
\end{thmintro}

In combination with the aforementioned properties of the involved affine Hecke
algebras, Theorem \ref{thm:D} provides the bijection in Theorem \ref{thm:A}.
Most of further properties mentioned in our main theorem follow rather quickly
from earlier work on such algebras \cite{AMS2,AMS3,SolComp}. \\

A few properties which can be expected of a local Langlands correspondence remain
open in Theorem \ref{thm:A}. Comparing with Borel's list of desiderata in \cite[\S 10]{Bor},
one notes that we have shown all of them, except for the functoriality with respect to
homomorphisms of reductive groups with commutative kernel and commutative cokernel.
We believe that this holds in a sense which is more precise and stronger than the
formulation in \cite{Bor}, but the proof appears to be cumbersome.

Further, it would nice to establish the HII conjectures for all unipotent representations.
In \cite[\S 16]{FOS} that was done for supercuspidal unipotent representations,
and in \cite{Opd18} a weaker version was proven for all unipotent representations.

A rather ambitious issue is the stability of the L-packets constructed in this paper.
Given $\phi \in \Phi_\nr (G)$, is there a linear combination of the members of the
L-packet $\Pi_\phi (G)$ whose trace gives a stable distribution on $G$? And if so, is
$\Pi_\phi (G)$ minimal for this property?

We hope to address these open problems in future work.

\section{Langlands dual groups and Levi subgroups}
\label{sec:Levi}

For more background on the material in this section, cf. \cite[\S 1--3]{Bor} and 
\cite[\S 2]{SiZi}. Let $K$ be field with an algebraic closure $\overline K$ and a 
separable closure $K_s \subset \overline K$. Let $\Gamma_K$ be a dense subgroup of the 
Galois group of $K_s / K$, for example Gal$(K_s / K)$ or, when $K$ is local and 
nonarchimedean, the Weil group of $K$.

Let $\mc G$ be a connected reductive $K$-group. Let $\mc T$ be a maximal torus 
of $\mc G$, and let $\Phi (\mc G, \mc T)$ 
be the associated root system. We also fix a Borel subgroup $\mc B$ of $\mc G$ 
containing $\mc T$, which determines a basis $\Delta$ of $\Phi (\mc G, \mc T)$. For 
every $\gamma \in \Gamma_K$ there exists a $g_\gamma \in \mc G (K_s)$ such that
\[
g_\gamma \gamma (\mc T) g_\gamma^{-1} = \mc T \quad \text{and} \quad
g_\gamma \gamma (\mc B) g_\gamma^{-1} = \mc B . 
\]
One defines an action of $\Gamma_K$ on $\mc T$ by
\begin{equation}\label{eq:1.2}
\mu_{\mc B}(\gamma) (t) = \mr{Ad}(g_\gamma) \circ \gamma (t) .
\end{equation}
This also determines an action $\mu_{\mc B}$ of $\Gamma_K$ on $\Phi (\mc G,\mc T)$,
which stabilizes $\Delta$.

Let $\Phi (\mc G,\mc T)^\vee$ be the dual root system of $\Phi (\mc G,\mc T)$,
contained in the cocharacter lattice $X_* (\mc T)$. The based root datum of $\mc G$ is
\[
\big( X^* (\mc T), \Phi (\mc G,\mc T), X_* (\mc T), \Phi (\mc G,\mc T)^\vee, \Delta \big) .
\]
Let $\mc S$ be a maximal $K$-split torus in $\mc G$. By \cite[Theorem 13.3.6.(i)]{Spr} 
applied to $Z_{\mc G}(\mc S)$, we may assume that $\mc T$ is defined over $K$ and 
contains $\mc S$. Then $Z_{\mc G}(\mc S)$ is a minimal $K$-Levi subgroup 
of $\mc G$. Let
\[
\Delta_0 := \{ \alpha \in \Delta : \mc S \subset \ker \alpha \}
\]
be the set of simple roots of $(Z_{\mc G}(\mc S), \mc T)$. 

Recall from \cite[Lemma 15.3.1]{Spr} that the root system $\Phi (\mc G, \mc S)$ is the 
image of $\Phi (\mc G, \mc T)$ in $X^* (\mc S)$, without 0. The set of simple roots of
$(\mc G, \mc S)$ can be identified with $(\Delta \setminus \Delta_0 ) / \mu_{\mc B}(\Gamma_K)$.
The Weyl group of $(\mc G, \mc S)$ can be expressed in various ways:
\begin{equation}\label{eq:1.1}
\begin{aligned}
W(\mc G,\mc S) & = N_{\mc G}(\mc S) / Z_{\mc G}(\mc S) \cong
N_{\mc G (K)}(\mc S(K)) / Z_{\mc G (K)}(\mc S (K)) \\
& \cong N_{\mc G}(\mc S,\mc T) / N_{Z_{\mc G}(\mc S)}(\mc T) =
\big( N_{\mc G}(\mc S,\mc T) / \mc T \big) \big/ \big( N_{Z_{\mc G}(\mc S)}(\mc T) / \mc T \big) \\
& \cong \mr{Stab}_{W(\mc G,\mc T)} (\mc S) / W(Z_{\mc G}(\mc S), \mc T) . 
\end{aligned}
\end{equation}
Let $\mc P_{\Delta_0} = Z_{\mc G}(\mc S) \mc B$ the minimal parabolic $K$-subgroup of $\mc G$ 
associated to $\Delta_0$. It is well-known \cite[Theorem 15.4.6]{Spr} 
that the following sets are canonically in bijection:
\begin{itemize}
\item $\mc G (K)$-conjugacy classes of parabolic $K$-subgroups of $\mc G$;
\item standard (i.e. containing $\mc P_{\Delta_0}$) parabolic $K$-subgroups of $\mc G$;
\item subsets of $(\Delta \setminus \Delta_0 ) /  \mu_{\mc B}(\Gamma_K)$;
\item $ \mu_{\mc B}(\Gamma_K)$-stable subsets of $\Delta$ containing $\Delta_0$.
\end{itemize}
By \cite[Lemma 15.4.5]{Spr} every $ \mu_{\mc B}(\Gamma_K)$-stable subset $I \subset \Delta$ containing
$\Delta_0$ gives rise to a standard Levi $K$-subgroup $\mc L_I$ of $\mc G$, namely the
group generated by $Z_{\mc G}(\mc S)$ and the root subgroups for roots in
$\Z I \cap \Phi (\mc G,\mc T)$.
The following description of conjugacy classes of Levi $K$-subgroups of $\mc G$ is
undoubtedly known, we provide the proof because we could not find it in the literature.

\begin{lem}\label{lem:1.1}
\enuma{
\item Every Levi $K$-subgroup of $\mc G$ is $\mc G (K)$-conjugate to a standard Levi
$K$-subgroup of $\mc G$.
\item For two standard Levi $K$-subgroups $\mc L_I$ and $\mc L_J$ the following are
equivalent:
\begin{enumerate}[(i)]
\item $\mc L_I$ and $\mc L_J$ are $\mc G (K)$-conjugate;
\item $(I \setminus \Delta_0) /  \mu_{\mc B}(\Gamma_K)$ and $(J \setminus \Delta_0) /  \mu_{\mc B}(\Gamma_K)$
are $W (\mc G, \mc S)$-associate.
\end{enumerate} }
\end{lem}
\begin{proof}
(a) Let $\mc P$ be a parabolic $K$-subgroup of $\mc G$ with a Levi factor $\mc L$
defined over $K$. Since $\mc P$ is $\mc G (K)$-conjugate to a standard parabolic
subgroup $\mc P_I$ \cite[Theorem  15.4.6]{Spr}, $\mc L$ is $\mc G (K)$-conjugate
to a Levi factor of $\mc P_I$. By \cite[Proposition 16.1.1]{Spr} any two such factors
are conjugate by an element of $\mc P_I (K)$. In particular $\mc L$ is 
$\mc G (K)$-conjugate to $\mc L_I$.\\
(b) Suppose that (ii) is fulfilled, that is, 
\[
w(I \setminus \Delta_0) /  \mu_{\mc B}(\Gamma_K) = (J \setminus \Delta_0) /  \mu_{\mc B}(\Gamma_K)
\quad \text{for some } w \in W(\mc G,\mc S) .
\]
Let $\bar w \in N_{\mc G (K)}(\mc S(K))$ be a lift of $w$. Then $\bar w \mc L_I {\bar w}^{-1}$
contains $Z_{\mc G} (\mc S)$ and 
\[
\Phi (\bar w \mc L_I {\bar w}^{-1} ,\mc S) = w \Phi (\mc L_I,\mc S) = \Phi (\mc L_J ,\mc S) .
\]
Hence $\bar w \mc L_I {\bar w}^{-1} = \mc L_J$, showing that (i) holds.

Conversely, suppose that (ii) holds, so $g \mc L_I g^{-1} = \mc L_J$ for some $g \in 
\mc G (K)$. Then $g \mc S g^{-1}$ is a maximal $K$-split torus of $\mc L_J$.
By \cite[Theorem 15.2.6]{Spr} there is a $l \in \mc L_J (K)$ such that 
$l g \mc S g^{-1} l^{-1} = \mc S$. Thus $(lg) \mc L_I (l g)^{-1} = \mc L_J$ and
$lg \in N_{\mc G}(\mc S)$. Let $w_1$ be the image of $l g$ in $W(\mc G,\mc S)$. Then 
$w_1 (\Phi (\mc L_I,\mc S)) = \Phi (\mc L_J,\mc S)$, so $w_1 \big( (I \setminus \Delta_0) 
/ \mu_{\mc B}(\Gamma_K) \big)$ is a basis of $\Phi (\mc L_J,\mc S)$. Any two bases 
of a root system are associate under its Weyl group, so there exists a 
$w_2 \in W(\mc L_J,\mc S) \subset W(\mc G,\mc S)$ such that 
\[
w_2 w_1 \big( (I \setminus \Delta_0) /  \mu_{\mc B}(\Gamma_K) \big) = 
(J \setminus \Delta_0) / \mu_{\mc B}(\Gamma_K) . \qedhere
\]
\end{proof}

Let $\mc G^\vee$ be the split reductive group with based root datum
\[
\big( X_* (\mc T), \Phi (\mc G,\mc T)^\vee, X^* (\mc T), 
\Phi (\mc G,\mc T), \Delta^\vee \big) . 
\]
Then $G^\vee = \mc G^\vee (\C)$ is the complex dual group of $\mc G$. Via the choice of
a pinning, the action $\mu_{\mc B}$ of $\Gamma_K$ on the root datum of $\mc G$, from 
\eqref{eq:1.2}, determines an action of $\Gamma_K$ of $G^\vee$. That action stabilizes
the torus $T^\vee = X^* (\mc T) \otimes_\Z \C^\times$ and the Borel subgroup
$B^\vee$ determined by $T^\vee$ and $\Delta^\vee$. The Langlands dual group (in the 
version based on $\Gamma_K$) of $\mc G (K)$ is ${}^L G := G^\vee \rtimes \Gamma_K$.

Every subset $I \subset \Delta$ corresponds to a unique subset $I^\vee \subset \Delta^\vee$,
and as such gives rise to a standard parabolic subgroup $P_I^\vee \subset G^\vee$ and a
standard Levi subgroup $L_I^\vee$. Following \cite{Bor,AMS1}, we define a L-parabolic 
subgroup ${}^L P$ of ${}^L G$ to be the normalizer of a parabolic subgroup
$P^\vee \subset G^\vee$ for which the canonical map
$N_{G^\vee \rtimes \Gamma_K}(P^\vee) \to \Gamma_K$ is surjective. As $\Gamma_K \subset
\mr{Gal}(K_s / K)$ is totally disconnected, $({}^L P )^\circ = P^\vee$.

Let $T_L^\vee \subset G^\vee$ be a torus such that $Z_{G^\vee \rtimes \Gamma_K}(T_L^\vee) \to
\Gamma_K$ is surjective. Then we call $Z_{G^\vee \rtimes \Gamma_K}(T_L^\vee)$ a
Levi L-subgroup of ${}^L G$. Notice that $( Z_{G^\vee \rtimes \Gamma_K}(T_L^\vee) )^\circ
= Z_{G^\vee}(T_L^\vee)$ is a Levi subgroup of $G^\vee$.

Special cases include $P_I^\vee \rtimes \Gamma_K$ and $L_I^\vee \rtimes \Gamma_K$, where
$P_I^\vee$ (resp. $L_I^\vee$) is a standard Levi subgroup of $G^\vee$ such that
$I$ is $\Gamma_K$-stable. We call these standard L-parabolic (resp. L-Levi) subgroups
of ${}^L G$.

We say that a L-parabolic (resp. L-Levi) subgroup ${}^L H \subset {}^L G$ is 
$\mc G (K)$-relevant if the $G^\vee$-conjugacy class of $({}^L H)^\circ \subset G^\vee$
corresponds to a conjugacy class of parabolic (resp. Levi) $K$-subgroups of $\mc G$.
As observed in \cite[\S 3]{Bor}, for $\Gamma_K$-stable $I \subset \Delta$ 
\begin{equation}\label{eq:1.5}
\text{the groups } P_I^\vee \rtimes \Gamma_K \text{ and } L_I^\vee \rtimes \Gamma_K 
\text{ are } \mc G (K)\text{-relevant if and only if } \Delta_0 \subset I . 
\end{equation}
Moreover the correspondence
\begin{equation}\label{eq:1.4}
\mc P_I \longleftrightarrow P_I^\vee \rtimes \Gamma_K 
\end{equation}
provides a bijection between the set of $\mc G (K)$-conjugacy classes of parabolic
$K$-subgroups of $\mc G$ and the set of $G^\vee$-conjugacy classes of $\mc G(K)$-relevant
L-parabolic subgroups of ${}^L G$ \cite[\S 3]{Bor}. Similarly, there is a bijective
correspondence between the set of standard Levi $K$-subgroups of $\mc G$ and the set of
standard $\mc G (K)$-relevant L-Levi subgroups of ${}^L G$:
\begin{equation}\label{eq:1.6}
\mc L_I \longleftrightarrow L_I^\vee \rtimes \Gamma_K .
\end{equation}
The actions of $\Gamma_K$ on $\Phi (\mc G,\mc T)$ and on $\Phi (\mc G,\mc T)^\vee = 
\Phi (G^\vee,T^\vee)$ induce $\Gamma_K$-actions on the associated Weyl groups. The 
$\Gamma_K$-equivariant isomorphism 
\[
W(\mc G,\mc T) \cong W(G^\vee,T^\vee)
\]
can be modified to a version for $\mc S$. Namely, it was shown in 
\cite[Proposition 3.1 and (43)]{ABPSLLC} that there are canonical isomorphisms
\begin{equation}\label{eq:1.3}
W(\mc G,\mc S) \to \mr{Stab}_{W(G^\vee,T^\vee)^{\Gamma_K}}(\Z \Delta_0^\vee) \big/
W (L_{\Delta_0}^\vee, T^\vee)^{\Gamma_K} \to 
N_{G^\vee}(L_{\Delta_0}^\vee \rtimes \Gamma_K) \big/ L_{\Delta_0}^\vee .
\end{equation}
As $W(G^\vee,T^\vee)$ acts naturally on $X_* (\mc T) = X^* (T^\vee)$,
$\mr{Stab}_{W(G^\vee,T^\vee)^{\Gamma_K}}(\Z \Delta_0^\vee)$ acts on $X_* (\mc T) / 
\Z \Delta_0^\vee$. This descends to a natural action of 
\[
\mr{Stab}_{W(G^\vee,T^\vee)^{\Gamma_K}}(\Z \Delta_0^\vee) \big/
W (L_{\Delta_0}^\vee, T^\vee)^{\Gamma_K} \quad \text{on} \quad 
X_* (\mc T) / \Z \Delta_0^\vee,
\]
which stabilizes the image of $\Phi (\mc G,\mc T)^\vee$ in $X_* (\mc T) / \Z \Delta_0^\vee$.
As observed in \cite[Proposition 2.5.4]{SiZi}, the correspondences \eqref{eq:1.4} and 
\eqref{eq:1.6} are $W(\mc G ,\mc S)$-equivariant, with respect to \eqref{eq:1.3}.

\begin{lem}\label{lem:1.2}
\enuma{
\item Every $\mc G (K)$-relevant L-Levi subgroup of ${}^L G$ is $G^\vee$-conjugate to
a $\mc G (K)$-relevant standard L-Levi subgroup of ${}^L G$.
\item Let $I^\vee, J^\vee \subset \Delta^\vee$ be $\Gamma_K$-stable subsets containing
$\Delta_0^\vee$. The two $\mc G(K)$-relevant standard L-Levi subgroups 
$L_I^\vee \rtimes \Gamma_K$ and $L_J^\vee \rtimes \Gamma_K$ are $G^\vee$-conjugate if and only 
if there exists a $w^\vee \in \mr{Stab}_{W(G^\vee,T^\vee)^{\Gamma_K}}(\Z \Delta_0^\vee) /
W (L_{\Delta_0}^\vee, T^\vee)^{\Gamma_K}$ with 
$w^\vee (I^\vee \setminus \Delta_0 ) = J^\vee \setminus \Delta_0$.
} 
\end{lem}
\begin{proof}
(a) By \cite[Lemma 6.2]{AMS1} every L-Levi subgroup of ${}^L G$ is $G^\vee$-conjugate
to a standard L-Levi subgroup. By definition $G^\vee$-conjugacy preserves
$\mc G (K)$-relevance.\\
(b) Suppose that a $w^\vee$ with the indicated properties exists. Let $\tilde w \in
N_{G^\vee}(T^\vee)$ be a lift of $w^\vee$ (by \eqref{eq:1.3} it is unique up to
$N_{L_{\Delta_0}^\vee}(T^\vee)$). Then $\tilde w (L_I^\vee \rtimes \Gamma_K) {\tilde w}^{-1}$
contains $L_{\Delta_0}^\vee \rtimes \Gamma_K$ and the roots of 
\[
\big( \tilde w (L_I^\vee \rtimes \Gamma_K) {\tilde w}^{-1}\big)^\circ 
= \tilde w L_I^\vee {\tilde w}^{-1}
\]
with respect to $T^\vee$ are 
\[
w (\Phi (L_I^\vee,T^\vee)) = w (\Z I^\vee \cap \Phi (G^\vee,T^\vee)) =
\Z J^\vee \cap \Phi (G^\vee,T^\vee) = \Phi (L_J^\vee,T^\vee) .
\]
Hence $\tilde w (L_I^\vee \rtimes \Gamma_K) {\tilde w}^{-1} = L_J^\vee$.

Conversely, suppose that $g (L_I^\vee \rtimes \Gamma_K) g^{-1} = L_J^\vee \rtimes \Gamma_K$
for some $g \in G^\vee$. Then
\[
L_J^\vee = (L_J^\vee \rtimes \Gamma_K )^\circ = \big( g (L_I^\vee \rtimes \Gamma_K) 
g^{-1} \big)^\circ = g L_I^\vee g^{-1} .
\]
In the proof of Lemma \ref{lem:1.1}.b we showed that there exists a $l_1 \in L_J^\vee$
such that $l_1 g \in N_{G^\vee}(T^\vee)$ and $(l_1 g) L_I^\vee \rtimes \Gamma_K (l_1 g)^{-1}
= L_J^\vee \rtimes \Gamma_K$. Now $(l_1 g) (L_I^\vee \cap B^\vee) (l_1 g)^{-1}$ is a Borel
subgroup of $L_J^\vee$ containing $T^\vee$. By the conjugacy of Borel subgroups and maximal
tori, there exists $l_2 \in N_{L_J^\vee}(T^\vee)$ such that 
\[
l_2 l_1 g (L_I^\vee \cap B^\vee) g^{-1} l_1^{-1} l_2^{-1} = L_J^\vee \cap B^\vee.  
\]
Then conjugation by $l_2 l_1 g$ sends the set $\Delta_0$ of simple roots for $L_{\Delta_0}^\vee$
to the set of simple roots for $L_J^\vee \cap B^\vee$. Hence
\[
l_2 l_1 g (L_{\Delta_0}^\vee \rtimes \Gamma_K) g^{-1} l_1^{-1} l_2^{-1}
\; \subset \; L_J^\vee \rtimes \Gamma_K
\]
is a standard L-Levi subgroup of ${}^L G$. It is conjugate to $L_{\Delta_0}^\vee \rtimes \Gamma_K$, 
so $\mc G (K)$-relevant and minimal for that property. As $L_{\Delta_0}^\vee \rtimes \Gamma_K$
is the unique standard minimal $\mc G(K)$-relevant L-Levi subgroup of ${}^L G$, it must
be normalized by $l_2 l_1 g$. Thus $l_2 l_1 g \in N_{G^\vee}(L_{\Delta_0}^\vee \rtimes \Gamma_K,
T^\vee)$ sends $I^\vee$ to $J^\vee$. By \eqref{eq:1.3} there exists a
$w^\vee \in \mr{Stab}_{W(G^\vee,T^\vee)^{\Gamma_K}}(\Z \Delta_0^\vee)$ mapping to $l_2 l_1 g$,
and then $w^\vee (I^\vee) = J^\vee$.
\end{proof}

By \eqref{eq:1.3} the set of orbits of $W(\mc G,\mc S)$ on $(\Delta \setminus \Delta_0) /
\mu_{\mc B}(\Gamma_K)$ is canonically in bijection with the set of orbits of  
$\mr{Stab}_{W(G^\vee,T^\vee)^{\Gamma_K}}(\Z \Delta_0^\vee) \big/
W (L_{\Delta_0}^\vee, T^\vee)^{\Gamma_K}$ on $(\Delta^\vee \setminus \Delta_0^\vee ) /
\Gamma_K$. This and Lemmas \ref{lem:1.1} and \ref{lem:1.2} yield the version of \eqref{eq:1.4} 
for Levi subgroups that we were after:

\begin{cor}\label{cor:1.3}
The assignment $\mc L_I \mapsto L_I^\vee \rtimes \Gamma_K$ from \eqref{eq:1.6} provides a 
bijection between the set of $\mc G (K)$-conjugacy classes of Levi $K$-subgroups of $\mc G$ 
and the set of $G^\vee$-conjugacy classes of $\mc G(K)$-relevant L-Levi subgroups of ${}^L G$.
\end{cor}

\section{Hecke algebras for Langlands parameters}
\label{sec:2}

From now on $K$ is a non-archimedean local field with ring of integers $\mf o_K$ and
a uniformizer $\varpi_K$. Let $k = \mf o_K / \varpi_K \mf o_K$ be its residue
field, of cardinality $q_K$. Let $\mb W_K \subset \mr{Gal}(K_s / K)$ be the Weil group of $K$ 
and let $\Fr$ be an (arithmetic) Frobenius element. Let $\mb I_K  \subset \mb W_K$ be the 
inertia subgroup, so that $\mb W_K / \mb I_K \cong \Z$ is generated by $\Fr$.

We let $\mc G$ and its subgroups be as in Section \ref{sec:Levi}.
We write $G = \mc G(K)$ and similarly for other $K$-groups. 
Recall that a Langlands parameter for $G$ is a homomorphism 
\[
\phi : \mb W_K \times SL_2 (\C) \to {}^L G = G^\vee \rtimes \mb W_K ,
\]
with some extra requirements. In particular $\phi |_{SL_2 (\C)}$ has to be algebraic, 
$\phi (\mb W_K)$ must consist of semisimple elements and $\phi$ must respect the
projections to $\mb W_K$. 

We say that a L-parameter $\phi$ for $G$ is 
\begin{itemize}
\item discrete if there does not exist any proper L-Levi subgroup of ${}^L G$ containing the
image of $\phi$;
\item bounded if $\phi (\Fr) = (s,\Fr)$ with $s$ in a bounded subgroup of $G^\vee$;
\item unramified if $\phi (w) = (1,w)$ for all $w \in \mb I_K$.
\end{itemize}
Let ${G^\vee}_\ad$ be the adjoint group of $G^\vee$, and let ${G^\vee}_\Sc$ be its
simply connected cover. Let $\mc G^*$ be the unique $K$-quasi-split inner form of $\mc G$. 
We consider $\mc G$ as an inner twist of $\mc G^*$, so endowed with a $K_s$-isomorphism
$\mc G \to \mc G^*$. Via the Kottwitz isomorphism $\mc G$ is labelled by character
$\zeta_{\mc G}$ of $Z({G^\vee}_\Sc)^{\mb W_K}$ (defined with respect to $\mc G^*$).

Both ${G^\vee}_\ad$ and ${G^\vee}_\Sc$ act on $G^\vee$ by conjugation. As
\[
Z_{G^\vee}(\text{im } \phi) \cap Z(G^\vee) = Z(G^\vee)^{\mb W_K} ,
\]
we can regard $Z_{G^\vee}(\text{im } \phi) / Z(G^\vee)^{\mb W_K}$ as a subgroup of 
${G^\vee}_\ad$.  Let $Z^1_{{G^\vee}_\Sc}(\text{im } \phi)$ be its inverse image in
${G^\vee}_\Sc$ (it contains $Z_{{G^\vee}_\Sc}(\text{im } \phi)$ with finite index). The
S-group of $\phi$ is
\[
\mc S_\phi := \pi_0 \big( Z^1_{{G^\vee}_\Sc}(\text{im } \phi) \big) .
\]
An enhancement of $\phi$ is an irreducible representation $\rho$ of $\mc S_\phi$.
Via the canonical map $Z({G^\vee}_\Sc)^{\mb W_K} \to \mc S_\phi$, $\rho$ determines
a character $\zeta_\rho$ of $Z({G^\vee}_\Sc)^{\mb W_K}$. 
We say that an enhanced L-parameter $(\phi,\rho)$ is relevant for $G$ if $\zeta_\rho =
\zeta_{\mc G}$. This is equivalent to $\phi$ being $G$-relevant in terms of Levi
subgroups \cite[Lemma 9.1]{HiSa}. In view of \eqref{eq:1.5}, this means that 
$(\phi,\rho)$ is $G$-relevant if and only if every L-Levi subgroup of ${}^L G$ 
containing the image of $\phi$ is $G$-relevant. The group $G^\vee$ acts naturally on 
the collection of $G$-relevant enhanced L-parameters, by 
\[
g \cdot (\phi,\rho) = (g \phi g^{-1},\rho \circ \mr{Ad}(g)^{-1}) .
\]
We denote the set of $G^\vee$-equivalence classes of $G$-relevant (resp. enhanced) 
L-parameters by $\Phi (G)$, resp. $\Phi_e (G)$. A local Langlands correspondence 
for $G$ (in its modern interpretation) should be a bijection between $\Phi_e (G)$ and 
the set of irreducible smooth $G$-representations, with several nice properties.

Let $H^1 (\mb W_K, Z(G^\vee))$ be the first Galois cohomology group of $\mb W_K$ with
values in $Z(G^\vee)$. It acts on $\Phi (G)$ by
\begin{equation}\label{eq:2.9}
(z \phi) (w,x) = z' (w) \phi (w,x) \qquad \phi \in \Phi (G), w \in \mb W_K, x \in SL_2 (\C) ,
\end{equation}
where $z' : \mb W_K \to Z(G^\vee)$ represents $z \in H^1 (\mb W_K, Z(G^\vee))$. 
This extends to an action of $H^1 (\mb W_K, Z(G^\vee))$ on $\Phi_e (G)$, which does 
nothing to the enhancements.

Let us focus on cuspidality for enhanced L-parameters \cite[\S 6]{AMS1}. Consider
\[
G^\vee_\phi := Z^1_{{G^\vee}_\Sc}( \phi |_{\mb W_K}),
\]
a possibly disconnected complex reductive group. Then $u_\phi := \phi \big( 1, 
\big( \begin{smallmatrix} 1 & 1 \\ 0 & 1 \end{smallmatrix} \big)\big)$ is a unipotent element 
of $(G^\vee_\phi )^\circ$ and $\mc S_\phi \cong \pi_0 (Z_{G^\vee_\phi}(u_\phi))$. We say that 
$(\phi,\rho) \in \Phi_e (G)$ is cuspidal if $\phi$ is discrete and $(u_\phi,\rho)$ is a 
cuspidal pair for $G^\vee_\phi$. The latter means that $(u_\phi,\rho)$ determines a 
$G^\vee_\phi$-equivariant cuspidal local system on the $(G^\vee_\phi)^\circ$-conjugacy class 
of $u_\phi$. Notice that a L-parameter alone does not contain enough information to detect 
cuspidality, for that we really need an enhancement. Therefore we will often say 
"cuspidal L-parameter" for an enhanced L-parameter which is cuspidal. 

The set of $G^\vee$-equivalence classes of $G$-relevant cuspidal L-parameters is denoted
$\Phi_\cusp (G)$. It is conjectured that under the LLC $\Phi_\cusp (G)$ corresponds to
the set of supercuspidal irreducible smooth $G$-representations.

The cuspidal support of any $(\phi,\rho) \in \Phi_e (G)$ is defined in \cite[\S 7]{AMS1}. 
It is unique up to $G^\vee$-conjugacy and consists of a $G$-relevant L-Levi subgroup
${}^L L$ of ${}^L G$ and a cuspidal L-parameter  $(\phi_v, q\epsilon)$ for ${}^L L$. 
By Corollary \ref{cor:1.3} this ${}^L L$ corresponds to a unique (up to $G$-conjugation)
Levi $K$-subgroup $\mc L$ of $\mc G$. This allows us to express the aforementioned 
cuspidal support map as
\begin{equation}\label{eq:2.1}
\mb{Sc}(\phi,\rho) = (\mc L (K), \phi_v, q \epsilon), \quad \text{where }
(\phi_v, q \epsilon) \in \Phi_\cusp (\mc L (K)) .
\end{equation}
It is conjectured that under the LLC this map should correspond to Bernstein's cuspidal
support map for irreducible smooth $G$-representations.

Sometimes we will be a little sloppy and write that $L = \mc L (K)$ is a Levi subgroup of $G$. 
Let $X_\nr (L)$ be the group of unramified characters $L \to \C^\times$.
As worked out in \cite[\S 3.3.1]{Hai}, it is naturally isomorphic to 
$(Z (L^\vee)^{\mb I_K} )_\Fr^\circ \subset H^1 (\mb W_K, Z(L^\vee))$. As such it acts on 
$\Phi_e (L)$ and on $\Phi_\cusp (L)$ by \eqref{eq:2.9}.
A cuspidal Bernstein component of $\Phi_e (L)$ is a set of the form
\[
\Phi_e (L)^{\mf s_L^\vee} := 
X_\nr (L) \cdot (\phi_L,\rho_L) \text{ for some } (\phi_L,\rho_L) \in \Phi_\cusp (L) . 
\]
The group $G^\vee$ acts on the set of cuspidal Bernstein components for all Levi subgroups
of $G$. The $G^\vee$-action is just by conjugation, but to formulate it precisely, more 
general L-Levi subgroups of ${}^L G$ are necessary. We prefer to keep those out of the notations, 
since we do not need them to get all classes up to equivalence. With that convention, we can 
define an inertial equivalence class for $\Phi_e (G)$ as 
\[
\mf s \text{ is the } G^\vee \text{-orbit of } 
(L, X_\nr (L) \cdot (\phi_L,\rho_L) ), \text{ where } (\phi_L,\rho_L) \in \Phi_\cusp (L).
\]
The underlying inertial equivalence class for $\Phi_e (L)$ is $\mf s_L^\vee = (L,X_\nr (L) \cdot
(\phi_L,\rho_L))$. Here it is not necessary to take the $L^\vee$-orbit, for 
$(\phi_L,\rho_L) \in \Phi_e (L)$ is fixed by $L^\vee$-conjugation.

We denote the set of inertial equivalence classes for $\Phi_e (G)$ by $\mf{Be}^\vee (G)$.
Every $\mf s^\vee \in \mf{Be}^\vee (G)$ gives rise to a Bernstein component in $\Phi^e (G)$
\cite[\S 8]{AMS1}, namely
\begin{equation}\label{eq:2.10}
\Phi_e (G)^{\mf s^\vee} = \{ (\phi,\rho) \in \Phi_e (G) : \mb{Sc}(\phi,\rho) \in \mf s^\vee \} .
\end{equation}
The set of such Bernstein components is also parametrized by $\mf{Be}^\vee (G)$, and forms a 
partition of $\Phi_e (G)$.

Notice that $\Phi_e (L)^{\mf s^\vee_L} \cong \mf s_L^\vee$ has a canonical topology, 
coming from the transitive action of $X_\nr (L)$. More precisely, let $X_\nr (L,\phi_L)$ 
be the stabilizer in $X_\nr (L)$ of $\phi_L$. Then the complex torus 
\[
T_{\mf s_L^\vee} := X_\nr (L) / X_\nr (L,\phi_L)
\] 
acts simply transitively on $\mf s_L^\vee$. This endows $\mf s_L^\vee$ with the structure of an 
affine variety. (There is no canonical group structure on $\mf s_L^\vee$ though, for that one 
still needs to choose a basepoint.) 

To $\mf s^\vee$ we associate a finite group $W_{\mf s^\vee}$, in many
cases a Weyl group. For that, we choose $\mf s_L^\vee = (L, X_\nr (L) \cdot (\phi_L,\rho_L) )$
representing $\mf s^\vee$ (up to isomorphism, the below does not depend on this choice). We define 
$W_{\mf s^\vee}$ as the stabilizer of $\mf s_L^\vee$ in $N_{G^\vee}(L^\vee \rtimes \mb W_K) / L^\vee$. 
In this setting we write $T_{\mf s^\vee}$ for $T_{\mf s_L^\vee}$. Thus $W_{\mf s^\vee}$ acts on 
$\mf s_L^\vee$ by algebraic automorphisms and on $T_{\mf s^\vee}$ by group automorphisms (but the 
bijection $T_{\mf s^\vee} \to \mf s_L^\vee$ need not be $W_{\mf s^\vee}$-equivariant). 

Next we quickly review the construction of an affine Hecke algebra from a Bernstein component
of enhanced Langlands parameters. We fix a basepoint $\phi_L$ for $\mf s_L^\vee$ as in
\cite[Proposition 3.9]{AMS3}, and use that to identify $\mf s_L^\vee$ with $T_{\mf s_L^\vee}$.
Consider the possibly disconnected reductive group
\[
G^\vee_{\phi_L} = Z^1_{{G^\vee}_\Sc} (\phi_L |_{\mb W_K}) .
\]
Let $L_c^\vee$ be the Levi subgroup of ${G^\vee}_\Sc$ determined by $L^\vee$. There
is a natural homomorphism 
\begin{equation}\label{eq:2.2}
Z(L_c^\vee)^{\mb W_K,\circ} \to X_\nr (L) \to T_{\mf s_L^\vee}
\end{equation}
with finite kernel \cite[Lemma 3.7]{AMS3}. Using that and \cite[Lemma 3.10]{AMS3}, 
$\Phi (G_{\phi_L}^\circ, Z(L_c^\vee)^{\mb W_K,\circ})$ gives rise to a reduced root system 
$\Phi_{\mf s^\vee}$ in $X^* (T_{\mf s^\vee})$. The coroot system $\Phi^\vee_{\mf s^\vee}$ 
is contained in $X_* (T_{\mf s^\vee})$. That gives a root datum $\mc R_{\mf s^\vee}$, whose
basis can still be chosen arbitrarily. 

The construction of label functions $\lambda$ and $\lambda^*$ for $\mc R_{\mf s^\vee}$ 
consists of several steps. The numbers $\lambda (\alpha), \lambda^* (\alpha) \in \Z_{\geq 0}$
will be defined for all roots $\alpha \in \Phi_{\mf s^\vee}$. First, we pick
$t \in (Z(L_c^\vee)^{\mb I_K})^\circ_\Fr$ such that the reflection $s_\alpha$ fixes 
$t \phi_L (\Fr)$. Then $q \alpha$ lies in $\Phi \big( (G^\vee_{t \phi_L})^\circ,
Z(L_c^\vee)^{\mb W_K,\circ} \big)$ for some $q \in \Q_{>0}$, and $\lambda (\alpha), 
\lambda^* (\alpha)$ are related $\Q$-linearly to the labels $c(q\alpha), c^* (q\alpha)$ 
for a graded Hecke algebra \cite[\S 1]{AMS3} associated to
\begin{equation}\label{eq:2.3}
(G^\vee_{t \phi_L})^\circ = Z_{{G^\vee}_\Sc}(t \phi_L (\mb W_K))^\circ, 
Z(L_c^\vee)^{\mb W_K,\circ}, u_{\phi_L} \text{ and } \rho_L .
\end{equation}
These integers $c(q\alpha), c^* (q\alpha)$ were defined in \cite[Propositions 2.8, 2.10 and 
2.12]{LusCusp}, in terms of the adjoint action of $\log (u_{\phi_L})$ on
\[
\mr{Lie} (G^\vee_{t \phi_L})^\circ = \mr{Lie} \big( Z_{{G^\vee}_\Sc}(t \phi_L (\mb W_K)) \big) .
\]
In \cite[Proposition 3.11 and Lemma 3.12]{AMS3} it is described which 
$t \in (Z(L_c^\vee)^{\mb I_K})^\circ_\Fr$ we need to determine all labels: for each $\alpha 
\in \Phi_{\mf s^\vee}$ just one with $\alpha (t) = 1$, and sometimes one with $\alpha (t) = -1$. 

Finally, we choose an array $\vec{v}$ of nonzero complex numbers, one $v_j$ for every 
irreducible component of $\Phi_{\mf s^\vee}$. To these data one can attach an affine Hecke 
algebra $\mc H (\mc R_{\mf s^\vee}, \lambda, \lambda^*, \vec{v})$, as in \cite[\S 2]{AMS3}. 

The group $W_{\mf s^\vee}$ acts on $\Phi_{\mf s^\vee}$ and contains the Weyl group 
$W_{\mf s^\vee}^\circ$ of that root system. It admits a semidirect factorization
\[
W_{\mf s^\vee} = W_{\mf s^\vee}^\circ \rtimes \mf R_{\mf s^\vee} ,
\]
where $\mf R_{\mf s^\vee}$ is the stabilizer of a chosen basis of $\Phi_{\mf s^\vee}$.

Using the above identification of $T_{\mf s^\vee}$ with $\mf s_L^\vee$, we can reinterpret 
$\mc H (\mc R_{\mf s^\vee}, \lambda, \lambda^*, \vec{v})$ as an algebra 
$\mc H (\mf s_L^\vee, W_{\mf s^\vee}^\circ, \lambda, \lambda^*,\vec{v})$ 
whose underlying vector space is $\mc O (\mf s_L^\vee) \otimes \C [W_{\mf s^\vee}^\circ]$.
The group $\mf R_{\mf s^\vee}$ acts naturally on the based root datum $\mc R_{\mf s^\vee}$, 
and hence on \\
$\mc H (\mf s_L^\vee, W_{\mf s^\vee}^\circ, \lambda, \lambda^*,\vec{v})$ by algebra 
automorphisms \cite[Proposition 3.13.a]{AMS3}. From \cite[Proposition 3.13.b]{AMS3} 
we get a 2-cocycle $\natural : \mf R_{\mf s^\vee}^2 \to \C^\times$ and a twisted group 
algebra $\C[\mf R_{\mf s^\vee},\natural]$. Now we can define the twisted affine Hecke algebra
\begin{equation}\label{eq:2.4}
\mc H (\mf s^\vee, \vec{v}) := \mc H (\mf s_L^\vee, W_{\mf s^\vee}^\circ, \lambda, 
\lambda^*, \vec{v}) \rtimes \C[\mf R_{\mf s^\vee},\natural] .
\end{equation}
Up to isomorphism it depends only on $\mf s^\vee$ and $\vec{v}$ \cite[Lemma 3.14]{AMS3}.

The multiplication relations in $\mc H (\mf s^\vee, \vec{v})$ are based on the Bernstein
presentation of affine Hecke algebras, let us make them explicit. The vector space
$\C [W_{\mf s^\vee}^\circ] \subset \mc H (\mf s^\vee, \vec{v})$ is the Iwahori--Hecke algebra
$\mc H (W_{\mf s^\vee}^\circ, \vec{v}^{\, 2 \lambda})$, where $\vec{v}^{\, \lambda} (\alpha) 
= v_j^{\lambda (\alpha)}$ for the entry $v_j$ of $\vec{v}$ specified by $\alpha$. 
The conjugation action of $\mf R_{\mf s^\vee}$ on $W_{\mf s^\vee}^\circ$ induces an
action on $\mc H (W_{\mf s^\vee}^\circ, \vec{v}^{\, 2 \lambda})$.

The vector space $\mc O (\mf s_L^\vee)$ is embedded in $\mc H (\mf s^\vee, \vec{v})$ as
a maximal commutative subalgebra. The group $W_{\mf s^\vee}$ acts on it via its action
of $\mf s_L^\vee$, and every root $\alpha \in \Phi_{\mf s^\vee} \subset X^* (T_{\mf s^\vee})$ 
determines an element $\theta_\alpha \in \mc O (\mf s_L^\vee )^\times$, which does not
depend on the choice of the basepoint $\phi_L$ of $\mf s_L^\vee$ by 
\cite[Proposition 3.9.b]{AMS3}. For $f \in \mc O (\mf s_L^\vee)$
and a simple reflection $s_\alpha \in W_{\mf s^\vee}^\circ$ the following version of the
Bernstein--Lusztig--Zelevinsky relation holds:
\[
f N_{s_\alpha} - N_{s_\alpha} f = \big( (\mb z_j^{\lambda (\alpha)} \! - \mb z_j^{-\lambda (\alpha)}) 
+ \theta_{-\alpha} (\mb z_j^{\lambda^* (\alpha)} \! - \mb z_j^{-\lambda^* (\alpha)}) \big) 
(f - s_\alpha \cdot f) / (1 - \theta^{\, 2}_{-\alpha}) .
\]
Thus $\mc H (\mf s^\vee, \vec{v})$ depends on the following objects:
$\mf s_L^\vee, W_{\mf s^\vee}$ and the simple reflections therein, the label functions
$\lambda, \lambda^*$, the parameters $\vec{v}$ and the functions $\theta_\alpha :
\mf s_L^\vee \to \C^\times$ for reduced roots $\alpha \in \Phi_{\mf s^\vee}$. When
$W_{\mf s^\vee} \neq W_{\mf s^\vee}^\circ$, we also need the 2-cocycle $\natural$ on
$\mf R_{\mf s^\vee}$.

As in \cite[\S 3]{Lus-Gr}, the above relations entail that the centre of 
$\mc H (\mf s^\vee, \vec{v})$ is $\mc O (\mf s_L^\vee)^{W_{\mf s^\vee}}$. In other words,
the space of central characters for $\mc H (\mf s^\vee, \vec{v})$-representations
is $\mf s_L^\vee / W_{\mf s^\vee}$.

We note that when $\mf s^\vee$ is cuspidal,
\begin{equation}\label{eq:2.11}
\mc H (\mf s^\vee, \vec{z}) = \mc O (\mf s^\vee)
\end{equation}
and every element of $\mf s^\vee$ determines a character of $\mc H (\mf s^\vee,\vec{v})$.

The main reason for introducing $\mc H (\mf s^\vee, \vec{v})$ is the next result.
(See \cite[Definition 2.6]{AMS3} for the definition of tempered and essentially 
discrete series representations.)

\begin{thm}\label{thm:2.1} \textup{\cite[Theorem 3.16]{AMS3}} \\
Let $\mf s^\vee$ be an inertial equivalence class for $\Phi_e (G)$ and assume that the 
parameters $\vec{v}$ lie in $\R_{>1}$. Then there exists a canonical bijection
\[
\begin{array}{ccc}
\Phi_e (G)^{\mf s^\vee} & \to & \Irr (\mc H (\mf s^\vee, \vec{v})) \\
(\phi,\rho) & \mapsto & \bar M (\phi,\rho,\vec{v}) 
\end{array}
\]
with the following properties.
\begin{itemize}
\item $\bar M (\phi,\rho,\vec{v})$ is tempered if and only if $\phi$ is bounded.
\item $\phi$ is discrete if and only if $\bar M (\phi,\rho,\vec{v})$ is essentially
discrete series and the rank of $\Phi_{\mf s^\vee}$ equals 
$\dim_\C (T_{\mf s^\vee} / X_\nr (G))$.
\item The central character of $\bar M (\phi,\rho,\vec{v})$ is the product of
$\phi (\Fr)$ and a term depending only on $\vec{v}$ and a cocharacter associated to $u_\phi$.
\item Suppose that $\mb{Sc}(\phi,\rho) = (L, \chi_L \phi_L, \rho_L)$, where 
$\chi_L \in X_\nr (L)$. Then $\bar M (\phi,\rho,\vec{v})$ is a constituent of 
$\mr{ind}_{\mc H (\mf s^\vee_L, \vec{v})}^{\mc H (\mf s^\vee, \vec{v})} 
(L, \chi_L \phi_L,\rho_L)$.
\end{itemize}
\end{thm}

The irreducible module $\overline{M} (\phi,\rho,\vec{v})$ in Theorem \ref{thm:2.1}
is a quotient of a ``standard module" $\overline{E} (\phi,\rho,\vec{v})$, also studied
in \cite[Theorem 3.15]{AMS3}. By \cite[Lemma 3.16.a]{AMS3} every such standard module
is a direct summand of a module obtained by induction from a standard module associated
to a discrete enhanced L-parameter for a Levi subgroup of $G$.

The action of $H^1 (\mb W_K, Z(G^\vee))$ on $\Phi_e (G)$ commutes with that of its
subgroup $X_\nr (G)$, so it induces an action on $\mf{Be}^\vee (G)$. For 
$z \in H^1 (\mb W_K, Z(G^\vee))$ we write that as $\mf s^\vee \mapsto z \mf s^\vee$.
Since $z \phi_L$ differs from $\phi_L$ only by central elements (of $\mc G^\vee$), almost 
all data used to construct $\mc H (\mf s^\vee, \vec{z})$ are the same for $z \mf s^\vee$:
\[
T_{z \mf s^\vee} = T_{\mf s^\vee} ,\; W_{z \mf s^\vee} = W_{\mf s^\vee} \text{ and }
\Phi_{z \mf s^\vee} = \Phi_{\mf s^\vee} .
\]
Furthermore the objects $\lambda, \lambda^*, \natural$ for $\mf s^\vee$ and $z \mf s^\vee$
can be identified, and the action of $z$ gives a bijection $\mf s_L^\vee \to 
z \mf s_L^\vee$. Thus $z$ canonically determines an algebra isomorphism
\begin{equation}\label{eq:2.12}
\begin{array}{ccccc}
\mc H (z) : & \mc H (\mf s^\vee ,\vec{v}) & \to & \mc H (z \mf s^\vee, \vec{v}) & \\
 & f N_w & \mapsto & (f \circ z^{-1}) N_w & 
 \qquad f \in \mc O (\mf s_L^\vee), w \in W_{\mf s^\vee} .
\end{array}
\end{equation}
This defines a group action of $H^1 (\mb W_K, Z(G^\vee))$ on the algebra 
$\bigoplus_{\mf s^\vee \in \mf S^\vee} \mc H (\mf s^\vee, \vec{v})$, where $\mf S^\vee$
is a union of $H^1 (\mb W_K, Z(G^\vee))$-orbits in $\mf{Be}^\vee (G)$.

Composition with $\mc H (z)^{-1}$ gives a functor between module categories:
\[
z \otimes : \mr{Mod}(\mc H (\mf s^\vee, \vec{v})) \to 
\mr{Mod}(\mc H (z\mf s^\vee, \vec{v})) .
\]

\begin{lem}\label{lem:2.3}
\enuma{
\item The bijections from Theorem \ref{thm:2.1} are $H^1 (\mb W_K, Z(G^\vee))$-equivariant:
\[
\bar M (z \phi,\rho,\vec{v}) = z \otimes \bar M (\phi,\rho,\vec{v}) \qquad
(\phi,\rho) \in \Phi_e (G)^{\mf s^\vee}, z \in H^1 (\mb W_K, Z(G^\vee)) .
\]
\item The same holds for the standard modules from \cite[Theorem 3.15]{AMS3}:
\[
\bar E (z \phi,\rho,\vec{v}) = z \otimes \bar E (\phi,\rho,\vec{v}) \qquad
(\phi,\rho) \in \Phi_e (G)^{\mf s^\vee}, z \in H^1 (\mb W_K, Z(G^\vee)) .
\]
\item Suppose that $\phi$ is bounded and that $z \in H^1 (\mb W_K, Z(G^\vee))$. Then
\[
\bar M (z \phi,\rho,\vec{v}) = \bar E (z \phi,\rho,\vec{v}) .
\]
}
\end{lem}
\begin{proof}
(a) For $z \in X_\nr (G) \cong (Z(G^\vee)^{\mb I_K})^\circ_{\mb W_K}$ this was 
shown in \cite[Theorem 3.15.e]{AMS3}. For general $z$, Theorem \ref{thm:2.1} and the 
definition of $\mc H (z)^{-1}$ show that $\bar M (z \phi,\rho,\vec{v})$ and 
$z \otimes \bar M (\phi,\rho,\vec{v})$ have the same central character (an element of
$z \mf s_L^\vee / W_{\mf s^\vee}$). Then the complete analogy between the construction
of $\bar M (z \phi,\rho,\vec{v})$ and of $\bar M (\phi,\rho,\vec{v})$ in \cite{AMS3}
entails that $\bar M (z \phi,\rho,\vec{v}) = z \otimes \bar M (\phi,\rho,\vec{v})$. \\
(b) This can be shown in the same way as (a).\\
(c) For $z=1$ this is \cite[Theorem 3.15.f]{AMS3}. Apply parts (a) and (b) to that.
\end{proof}

Let us investigate the compatibility of Theorem \ref{thm:2.1} with suitable versions
of the Langlands classification. The Langlands classification for (extended) affine
Hecke algebras \cite[Corollary 2.2.5]{SolAHA} says, roughly, that every irreducible
module of $\mc H (\mf s^\vee,\vec{v})$ can be obtained from an irreducible tempered
module of a parabolic subalgebra, by first twisting with a strictly positive character,
then parabolic induction and subsequently taking the unique irreducible quotient.

Let $\phi \in \Phi (G)$ be arbitrary. The Langlands classification for L-parameters
\cite[Theorem 4.6]{SiZi} says that there exists a parabolic subgroup $P$ of $G$
with Levi factor $Q$, such that im$(\phi) \subset {}^L Q$ and $\phi$ can be written
as $z \phi_b$ with $\phi_b \in \Phi (Q)$ bounded and $Z \in X_\nr (Q)$ strictly
positive with respect to $P$. Furthermore $P$ is unique up to $G$-conjugation, and this
provides a bijection between L-parameters for $G$ and such triples $(P,\phi_b,z)$
considered up to $G$-conjugacy.

Let $\zeta$ be the character of $Z(\mc G^\vee_\Sc)$ determined by $\rho$, an extension 
of the character $\zeta_{\mc G} \in \Irr (Z(\mc G^\vee_\Sc)^{\mb W_F})$ which was used
to define $\mc G (F)$-relevance. Let $\zeta^{\mc Q} \in \Irr (Z(\mc Q^\vee_\Sc))$ be 
derived from $\zeta$ as in \cite[Lemma 7.4]{AMS1}. Let $p_\zeta \in \C [\mc S_\phi]$ and 
$p_{\zeta^{\mc Q}} \in \C [\mc S_\phi^{\mc Q}]$ be the central idempotents associated to 
these characters. By \cite[Theorem 7.10.b]{AMS1} there are natural isomorphisms 
\begin{equation}\label{eq:2.13}
p_{\zeta^{\mc Q}} \C [\mc S_{\phi_b}^{\mc Q}] = 
p_{\zeta^{\mc Q}} \C [\mc S_{z \phi_b}^{\mc Q}] \to p_\zeta \C [\mc S_{\phi}] .
\end{equation}
Hence $\phi$ and $\phi_b$ admit the same relevant enhancements.

\begin{prop}\label{prop:2.4}
Let $\vec{z} \in \R_{>1}^d, (\phi,\rho) \in \Phi_e (G)$ and let $P,Q$ be as above.
\enuma{
\item $\bar{E} (\phi,\rho,\vec{v}) \cong \mc H (\mf s^\vee,\vec{z}) 
\underset{\mc H (\mf s_Q^\vee, \vec{v})}{\otimes}  \bar{E}^Q (\phi,\rho,\vec{v})$.
\item $\bar{M} (\phi,\rho,\vec{v})$ is the unique irreducible quotient of
$\mc H (\mf s^\vee,\vec{z}) \underset{\mc H (\mf s_Q^\vee, \vec{v})}{\otimes} \bar{M}^Q 
(\phi,\rho,\vec{v}) \cong \bar{E} (\phi,\rho,\vec{v})$.
\item The $\mc H (\mf s_Q^\vee, \vec{v})$-module $\bar{M}^Q (\phi,\rho,\vec{v})$
is a twist of a tempered module by a character which is strictly positive with
respect to $P$.
}
\end{prop}
\begin{proof}
(a) In view of \eqref{eq:2.13}, the statement is an instance of \cite[Lemma 3.16.a]{AMS3}.
But to apply that lemma we need to check that its condition
\begin{equation}\label{eq:2.14}
\epsilon_{u_\phi,j}(z,\vec{v}) \neq 0  
\end{equation}
is fulfilled. Write $r_j = \log (v_j) \in \R_{>0}$. Via the definitions in 
\cite[pages 28 and 12]{AMS3} and using the notations from \cite[\S 1]{AMS3}, 
\eqref{eq:2.14} boils down to 
\begin{equation}\label{eq:2.15}
\det \Big( \mr{ad} \big( \matje{r_j}{0}{0}{-r_j} - \log (z_j) \big) - 2 r_j \Big) \neq 0 .
\end{equation}
Here we take the determinant of an endomorphism of a vector space defined in terms
of $P,Q, \log (u_\phi)$ and a semisimple factor $G^\circ_{\phi_b ,j}$ of $G_{\phi_b}$.
This brings us to the setting of modules for a graded Hecke algebra $\mh H^Q (r_j)$
with one parameter $r_j > 0$, associated a Levi subgroup of the complex group
$G_{\phi_b ,j}$. Using that $- \log (z_j)$ is strictly negative with respect to the
parabolic subgroup of $G_{\phi_b ,j}$ determined by $P$, it is not hard to verify
\eqref{eq:2.15} -- this was done in \cite[Lemma A.2]{AMS2}.\\
(b) By  part (a) and Lemma \ref{lem:2.3}.c 
\[
\bar{E} (\phi,\rho,\vec{v}) \cong \mc H (\mf s^\vee,\vec{z}) 
\underset{\mc H (\mf s_Q^\vee, \vec{v})}{\otimes} \bar{E}^Q (\phi,\rho,\vec{v}) =
\bar{E} (\phi,\rho,\vec{v}) \cong \mc H (\mf s^\vee,\vec{z}) 
\underset{\mc H (\mf s_Q^\vee, \vec{v})}{\otimes} \bar{M}^Q (\phi,\rho,\vec{v}) .
\]
Via the construction of $\bar{E} (\phi,\rho,\vec{v})$ in \cite[\S 3]{AMS3} and
\cite[Lemma 1.3 and Theorem 1.4]{AMS3}, we can reduce the statement about quotients
to modules over a graded Hecke algebra $\mh H (r_j)$ with one parameter $r_j$, 
associated to the complex group $G_{\phi_b ,j}$. Since $r_j = \log (z_j) \neq 0$, 
\cite[Theorem 3.20.a]{AMS2} applies, and says that $\bar{E} (\phi,\rho,\vec{v})$
has a unique irreducible quotient, namely $\bar{M} (\phi,\rho,\vec{v})$.\\
(c) By Lemma \ref{lem:2.3}.a  
\[
\bar{M}^Q (\phi,\rho,\vec{v}) = z \otimes \bar{M}^Q (\phi,\rho,\vec{v}) ,
\]
and by Theorem \ref{thm:2.1} $\bar{M}^Q (\phi,\rho,\vec{v})$ is tempered.
From \eqref{eq:2.12} we see that $z \otimes$ is a twist by a character which on
$P$-positive elements of $X^* (T_{\mf s^\vee_L})$ takes values in $\R_{>1}$.  
Here $P$-positive refers to those elements of $X^* (T_{\mf s^\vee_L})$ which lie in
the interior of the positive cone associated to the root system
$\Phi (Z_{\mc Q_\Sc^\vee} (\phi |_{\mb I_F})^\circ, T_{\mf s^\vee_L})$ (constructed
similarly as $\Phi_{\mf s^\vee}$) with the simple roots determined by $P$.
\end{proof}

Suppose that $K' / K$ is a finite extension inside the separable closure $K_s$. 
Suppose also that $\mc G'$ is a connected reductive $K'$-group and that 
$\mc G = \mr{Res}_{K' / K}(\mc G')$, the Weil restriction of $\mc G'$. Then 
\begin{equation}\label{eq:2.7}
\mc G (K) = \mc G' (K') \quad \text{and} \quad
\mc G^\vee = \mr{ind}_{\mb W_{K'}}^{\mb W_K} (\mc G'^\vee).
\end{equation}
According to \cite[Proposition 8.4]{Bor}, Shapiro's lemma yields a natural bijection
between L-parameters for the $K$-group $\mc G$ and for the $K'$-group $\mc G'$.
By \cite[Lemma A.3]{FOS} it extends naturally to a bijection
\begin{equation}\label{eq:2.5}
\Phi_e (\mc G (K)) \to \Phi_e (\mc G' (K')) , 
\end{equation}
which preserves cuspidality. For Levi $K$-subgroup $\mc L$ of $\mc G$ there is 
a unique Levi $K'$-subgroup $\mc L'$ of $\mc G'$ with $\mc L (K) = \mc L' (K')$.
By \cite[Proposition 3.1]{ABPSLLC} there are natural isomorphisms
\begin{equation}\label{eq:2.8}
N_{G'^\vee}(L'^\vee \rtimes \mb W_{K'}) / L'^\vee \cong W(\mc G', \mc L') \cong
W (\mc G, \mc L) \cong N_{G^\vee}(L^\vee \rtimes \mb W_K) / L^\vee .
\end{equation}
Applying \ref{eq:2.5} to all Levi subgroups of $\mc G (K)$ and invoking \eqref{eq:2.8}, 
we obtain a bijection
\begin{equation}\label{eq:2.6}
\mf{Be}^\vee (\mc G (K)) \to \mf{Be}^\vee (\mc G' (K')) :
\mf s^\vee \mapsto \mf s'^\vee .
\end{equation}

\begin{lem}\label{lem:2.2}
The algebra $\mc H (\mf s^\vee, \vec{v})$ from \eqref{eq:2.4} is invariant under 
Weil restriction of reductive groups. That is, in the above situation there is
a natural isomorphism $\mc H (\mf s^\vee, \vec{v}) \cong \mc H (\mf s'^\vee, \vec{v})$.
\end{lem}
\begin{proof}
Let $(t \phi_L, \rho_L) \in \Phi_\cusp (\mc L (K))$ be as in \eqref{eq:2.3}. Via
\eqref{eq:2.5} we can regard it also as a cuspidal L-parameter for a (unique) Levi 
subgroup $\mc L' (K') \subset \mc G' (K')$. From \eqref{eq:2.7} we see that
${G^\vee}_\Sc = \mr{ind}_{\mb W_{K'}}^{\mb W_K} ({G'^\vee}_\Sc)$ and 
\[
G'^\vee_{t \phi_L} = Z^1_{{G'^\vee}_\Sc} \big(t \phi_L \big|_{\mb W_{K'}} \big) =
Z^1_{{G^\vee}_\Sc} \big(t \phi_L \big|_{\mb W_K} \big) = G^\vee_{t \phi_L} .
\]
The entire construction of $\mc H (\mf s^\vee, \vec{v})$ in \cite[\S 3.2]{AMS3},
as recalled above, takes place in groups $G^\vee_{t \phi_L}$ for some
$t \in (Z(L_c^\vee)^{\mb I_K})^\circ_\Fr$. So if we start with $\mc G'$ and $K'$ 
instead of $\mc G$ and $K$, we end up with the same algebra.
\end{proof}

\section{Hecke algebras for unipotent representations}
\label{sec:3}

We preserve the setup from the previous sections. Since we will discuss unipotent 
representations, it is convenient to require (for the remainder of the paper) that $\mc G$ 
splits over an unramified extension of $K$. A large part of this section will be based on 
\cite{LusUni1,Mor2}. Although Lusztig works with split simple adjoint $K$-groups in 
\cite{LusUni1}, most of the first section of that paper holds just as well for our $\mc G$.

\subsection{Buildings, facets and associated groups} \

We denote the enlarged Bruhat--Tits building of $G$ by $\mc B (\mc G,K)$. It is the Cartesian
product of the semisimple Bruhat--Tits building $\mc{BT}(\mc G,K)$ and the vector space 
$X_* (Z_G (S)) \otimes_\Z \R$. Whereas $Z(G)$ acts trivially on $\mc{BT} (\mc G,K)$, only its 
maximal compact subgroup $Z (G)_\cpt$ fixes $\mc B (\mc G,K)$ pointwise. 

Let $\Sigma = \Phi \times \Z$ be the affine root system of $(\mc G,\mc S)$, which projects onto
the finite root system $\Phi = \Phi (\mc G,\mc S)$. Let $\mh A = X_* (\mc S) \otimes_\Z \R$
be the apartment of $\mc B (\mc G,K)$ associated to $\mc S$. Let $C_0$ be the unique
chamber in the positive Weyl chamber in $\mh A$ (determined by $\Delta$), whose closure 
contains $0$. Let $\Delta_\af$ be the set of simple affine roots in $\Sigma$ determined
by $C_0$. It contains $\Delta$ and one additional affine reflection for every simple
factor of $G$ which is not a torus and not anisotropic. The associated set of simple
affine reflections $S_\af$ generates an affine Weyl group $W_\af$. 
The standard Iwahori subgroup of $G$ is $P_{C_0}$ and the Iwahori--Weyl group of $(G,S)$ is
\begin{equation}\label{eq:3.31}
W := N_G (S) / (N_G (S) \cap P_{C_0}) \cong Z_G (S) / (Z_G (S) \cap P_{C_0}) \rtimes W(G,S) .
\end{equation}
We note that it acts on $\mh A$, with $W(G,S)$ acting linearly and $Z_G (S) / 
(Z_G (S) \cap P_{C_0})$ by translations. The kernel of this action is the finite subgroup
$Z_G (S)_\cpt / Z_{P_{C_0}}(S)$. 

Furthermore $W$ contains $W_\af$ as the subgroup supported on the kernel of the Kottwitz 
homomorphism for $G$. The group $\Omega := \{ w \in W : w (C_0) = C_0\}$ 
forms a complement to $W_\af$:
\begin{equation}\label{eq:3.10}
W = W_\af \rtimes \Omega .
\end{equation}
In particular $\Omega \cong W / W_\af$, which is isomorphic to the image of the Kottwitz
homomorphism for $G$, a subquotient of $\Irr (Z (G^\vee))$. This shows that $\Omega$ is abelian.

Every facet $\mf f$ of $\mc B (\mc G,K)$ is the Cartesian product of $X_* (Z_G (S)) \otimes_\Z \R$ 
and a facet in $\mc{BT}(\mc G,K)$. Let $P_{\mf f} \subset G$ be the parahoric subgroup 
associated to $\mf f$, and let $U_{\mf f}$ be its pro-unipotent radical. Then 
$\overline P_{\mf f} = P_{\mf f} / U_{\mf f}$ can be regarded as the $k$-points of a connected 
reductive group. More precisely, Bruhat and Tits
\cite{BrTi2} constructed an $\mf o_K$-model $\mc G_{\mf f}^\circ$ of $\mc G$, such that
$P_{\mf f} = \mc G_{\mf f}^\circ (\mf o_K)$. Then $\overline P_{\mf f}$ is the maximal reductive
quotient of $\mc G_{\mf f}^\circ (\mf o_K / \varpi_K \mf o_K)$. 
Let $\hat P_{\mf f}$ be the pointwise stabilizer of $\mf f$ in $G$. It contains $P_{\mf f}$ with
finite index, and $\hat P_{\mf f} / U_{\mf f}$ is the group of $k$-rational points of a
(possibly disconnected) reductive group. As $P_{\mf f}$ is a characteristic subgroup of
$\hat P_{\mf f}$, these two have the same normalizer in $G$.

Since $G$ acts transitively on the collection of chambers of $\mc B (\mc G,K)$, we may assume
without loss of generality 
that $\mf f$ is contained in the closure of $C_0$. Let $\Sigma_{\mf f}$ be the set of
affine roots that vanish on $\mf f$ and let $J := \Delta_\af \cap \Sigma_{\mf f}$ be its
subset of simple affine roots. The associated set of (affine) reflections $\{ s_j : j \in J\}$
generates a finite Weyl group $W_J$, which can be identified with the Weyl group of the
$k$-group $\overline P_{\mf f}$ (with respect to the torus $\mc S (k)$).

Let $\Phi^c_{\mf f}$ be the set of roots for $(G,S)$ that are constant on $\mf f$,
a parabolic root subsystem of $\Phi (G,S)$. Let $\mc L_{\mf f}$ be the Levi $K$-subgroup 
of $\mc G$ determined by $\mc S$ and $\Phi^c_{\mf f}$. By \cite[Theorem 2.1]{Mor2} 
\[
P_{L,\mf f} := P_{\mf f} \cap L_{\mf f}
\]
is a maximal parahoric subgroup of $L_{\mf f}$ and
\begin{equation}\label{eq:3.13}
\hat P_{L,\mf f} / P_{L,\mf f} = (\hat{P}_{\mf f} \cap L_{\mf f}) / (P_{\mf f} \cap L_{\mf f})
\cong \hat P_{\mf f} / P_{\mf f}. 
\end{equation}
Let $\Phi_{\mf f}$ be the image of $\Sigma_{\mf f}$ in $\Phi (G,S)$. Its closure $(\Q \Phi_{\mf f}) 
\cap \Phi (G,S)$ equals $\Phi^c_{\mf f}$. Although $\Phi^c_{\mf f}$ and $\Phi_{\mf f}$ have the 
same rank, it is quite possible that they have different Weyl groups. We write
\[
\Omega_{\mf f} = \{ \omega \in \Omega : \omega (\mf f) = \mf f \} =
\{ \omega \in \Omega : P_{\mf f} \omega P_{\mf f} \subset N_G (P_{\mf f}) \}
\cong N_G (P_{\mf f}) / P_{\mf f} .
\]
Since $\Omega$ is abelian (see the lines following\eqref{eq:3.10}), so is $\Omega_{\mf f}$.

Next we analyse a group that underlies a relevant Hecke algebra:
\begin{equation}\label{eq:3.2}
W(J,\sigma) := N_W (W_J) / W_J \cong \{ w \in W : w (J) = J \} ,
\end{equation}
This does not depend on $\sigma$, which will only be introduced later. We include $\sigma$
in the notation to comply with \cite{Mor2}. 

It is easy to see that the right hand side of \eqref{eq:3.2} contains $\Omega_{\mf f}$.
When $P_{\mf f}$ is a maximal parahoric subgroup of $G$, $W(J,\sigma)$ coincides with 
$\Omega_{\mf f}$. 
Otherwise $G$ has at least one simple factor such that $\Delta_\af \setminus J$ contains 
two vertices belonging to that factor. Let $\Delta_{\mf f, \af} \subset \Delta_\af \setminus J$
be the collection of all indices belonging to such simple factors of $G$. According to 
\cite[\S 1.18]{LusUni1}, every $i \in \Delta_{\mf f, \af}$ corresponds to a unique order two 
element $s_i \in W$, and we write $S_{\mf f,\af} = \{ s_i : i \in \Delta_{\mf f, \af} \}$.
This set generates an affine Weyl group $W_\af (J,\sigma)$ in $W(J,\sigma)$, and
\begin{equation}\label{eq:3.4}
W(J,\sigma) = W_\af (J,\sigma) \rtimes \Omega_{\mf f} .
\end{equation}
The Coxeter group $W_\af (J,\sigma)$ is the direct product of irreducible affine
Weyl groups, one for every simple factor of $G$ to which at least two elements of
$\Delta_\af \setminus J$ belong. Hence it can be written as
\begin{equation}\label{eq:3.7}
W_\af (J,\sigma) = X(J) \rtimes W^\circ (J,\sigma) ,
\end{equation}
where $X(J)$ is a lattice (in $X_* (Z_G (S)) \cong Z_G(S) / Z_G (S)_\cpt$) and 
$W^\circ (J,\sigma)$ is a finite Weyl group. The subgroup $X(J) \subset W_\af (J,\sigma)$
is canonically defined, namely as the set of elements whose conjugacy class is finite.

The set of simple reflections $S_{\mf f} = \{ s_i : i \in \Delta_{\mf f} \}$
for $W^\circ (J,\sigma)$ is a subset of $S_{\mf f,\af}$, such that the affine extension
of $\Delta_{\mf f}$ is $\Delta_{\mf f,\af}$. We note that 
\begin{multline}\label{eq:3.27}
|S_{\mf f}| = |\Delta_{\mf f}| = \mr{rk}(X(J)) = \dim (\mf f \cap \mc{BT}(\mc G,K)) 
= \mr{rk}_K (\mc L_{\mf f} / Z(\mc G)) = \mr{rk}\,\Phi (G,L_{\mf f}) . 
\end{multline} 
The set $\Delta_{\mf f} \cup J$ determines
a unique vertex $x_{\mf f}$ of $\bar{\mf f} \cap \mc{BT}(\mc G,K)$, and
\begin{equation}\label{eq:3.6}
W^\circ (J,\sigma) = \{ w \in W_\af (J,\sigma) : w (x_{\mf f}) = x_{\mf f} \} . 
\end{equation}
Let $\mh A_m \subset \mh A$ be the product of the standard apartments of those
simple factors of $G$ for which $\Delta_{\mf f}$ has just one element. Let 
$\mh A_{\mf f} \subset \mh A$ be the product of the standard apartments of the 
remaining simple factors of $G$ (those for which $\Delta_{\mf f}$ has more than one
element or which are isotropic tori). We have a $W$-stable decomposition
\begin{equation}\label{eq:3.9}
\mh A = \mh A_m \times \mh A_{\mf f} . 
\end{equation}
The group $\Omega_{\mf f}$ acts by conjugation on the normal subgroup $W_\af (J,\sigma)$
of $W(J,\sigma)$. This action stabilizes $S_{\mf f,\af}$ setwise, and the pointwise
stabilizer $\Omega_{\mf f}^1$ of $S_{\mf f,\af}$ consists of those $\omega \in \Omega_{\mf f}$ 
which fix the image of $\mf f$ in $\mc{BT}(\mc G,K)$ pointwise. Since $\Omega_{\mf f}$ 
is abelian and the centre of $W_\af (J,\sigma)$ is trivial (as for every affine Weyl group),
\begin{equation}\label{eq:3.11}
Z(W(J,\sigma)) = \Omega_{\mf f}^1 .
\end{equation}
Let $\Omega_{\mf f,\tor}$ be the pointwise stabilizer of $\mf f$ in 
$\Omega_{\mf f}$, a central subgroup of $W(J,\sigma)$. Since $W$ acts on $\mh A$ with
finite stabilizers, $\Omega_{\mf f,\tor}$ is finite. We note that
\begin{equation}\label{eq:3.14}
\hat P_{\mf f} / P_{\mf f} = \Omega_{\mf f,\tor} . 
\end{equation}

\begin{prop}\label{prop:3.2}
There exists a lattice $X_{\mf f}$ in $X_* (Z_G (S)) \cap \mh A_{\mf f}$ such that
\[
W(J,\sigma) / \Omega_{\mf f,\tor} = W_\af (J,\sigma) \rtimes 
\Omega_{\mf f} / \Omega_{\mf f,\tor} \cong W^\circ (J,\sigma) \ltimes X_{\mf f}.
\]
\end{prop}
\begin{proof}
Consider the action of any $\omega \in \Omega_{\mf f}$ on $\mh A_{\mf f}$. 
It can be written as $\omega^\circ \omega_t$, where
$\omega^\circ \in W(G,S) \rtimes \mh A_{\mf f}$ fixes $x_{\mf f} \times X_* (Z_G (S)) 
\otimes_\Z \R$ pointwise and $\omega_t \in \mh A_{\mf f}$ is a translation. 

We claim that $\omega^\circ \in W^\circ (J,\sigma)$. Since $G$ acts on 
$X_* (Z_G (S)) \otimes_\Z \R$ by translations and $W^\circ (J,\sigma)$ acts trivially
on that space, it suffices to prove this claim under the assumption that $\mc G$ is 
semisimple. Then the group $W$ is naturally a subgroup of the analogous group for the 
adjoint group of $\mc G$, and these two semisimple groups have the same 
$W_\af (J,\sigma)$. The adjoint group of $\mc G$ is a direct product of simple adjoint 
$K$-groups, so we may even assume that $\mc G$ is simple and adjoint. 

Now we are exactly in the setting of \cite[\S 1.20]{LusUni1}, and $\Omega_{\mf f,\tor}$ 
becomes the group denoted $\bar \Omega_1^u$ over there. Then $W(J,\sigma) / 
\Omega_{\mf f,\tor}$ underlies the ``arithmetic'' affine Hecke algebra in 
\cite[\S 1.18]{LusUni1}, with finite Weyl group $W^\circ (J,\sigma)$. By classification 
Lusztig showed in \cite[Theorem 6.3]{LusUni1} and \cite[Theorem 10.11]{LusUni2} that 
this algebra is isomorphic to a ``geometric'' affine Hecke algebra. By \cite[\S 4.1 and 
\S 5.12]{LusUni1} and \cite[\S 8.2]{LusUni2} such algebras admit a presentation in terms of 
root data, which means that every element of $W(J,\sigma) / \Omega_{\mf f,\tor}$ can be 
written as the product of a translation and an element of $W^\circ (J,\sigma)$. In 
particular $\omega^\circ \in W^\circ (J,\sigma)$, proving our claim.

The above argument also shows that $\omega^\circ$ and $\omega_t$ depend only on the
image of $\omega$ in $\Omega_{\mf f} / \Omega_{\mf f,\tor}$. Put
\begin{equation}\label{eq:3.29}
X_{\mf f} := X(J) \langle \omega_t : \omega \in \Omega_{\mf f} / \Omega_{\mf f,\tor} \rangle . 
\end{equation}
Every element of $X_{\mf f}$ gives the action of an element of $W(J,\sigma)$ on
$\mh A_{\mf f}$, so $X_{\mf f}$ embeds naturally in $X_* (Z_G (S)) \cap \mh A_{\mf f}$.
As a subgroup of a lattice, it is itself a lattice.

From the action of \eqref{eq:3.7} on $\mh A_{\mf f}$ and from 
$\omega \mapsto \omega^\circ \omega_t$ we get a surjective group homomorphism
\begin{equation}\label{eq:3.8}
W(J,\sigma) / \Omega_{\mf f,\tor} = W_\af (J,\sigma) \rtimes \Omega_{\mf f} /
\Omega_{\mf f,\tor} \longrightarrow W^\circ (J,\sigma) \ltimes X_{\mf f} .
\end{equation}
Suppose that its kernel is nontrivial, say it contains $w \omega$ with 
$w \in  W_\af (J,\sigma)$ and $\omega \in \Omega_{\mf f} / \Omega_{\mf f,\tor}$.
The homomorphism \eqref{eq:3.8} is injective on $W_\af (J,\sigma)$, for that
group acts trivially on the factor $\mh A_m$ from \eqref{eq:3.9}. So $\omega \neq 1$ 
and the action of $\omega$ on $\mh A_{\mf f}$ agrees with that of $w^{-1}$. 
By the definition of $\Omega_{\mf f}$, $\omega$ stabilizes $\overline{\mf f}$, whereas
the affine Weyl group $W_\af (J,\sigma)$ acts simply transitively on a chamber
complex with fundamental chamber $\overline{\mf f}$. Therefore $w$ must be trivial,
and $\omega$ lies in the kernel of \eqref{eq:3.8}. Then $\omega$ acts trivially on
$\mh A_{\mf f}$, so it only acts on the factor $\mh A_m$ of $\mh A$. As $\Omega_{\mf f}$
stabilizes $\mf f$, this means that $\omega$ fixes $\mf f$ pointwise. Thus
$\omega \in \Omega_{\mf f,\tor}$, and \eqref{eq:3.8} is injective.
\end{proof}

By Proposition \ref{prop:3.2} the centre of $W^\circ (J,\sigma) \ltimes X_{\mf f}$ is
the free abelian group $X_{\mf f}^{W^\circ (J,\sigma)}$. From that and \eqref{eq:3.11}
it follows that $\Omega_{\mf f,\tor}$ is precisely the torsion subgroup of\\
$Z ( W (J,\sigma))$, which justifies our notation.

\subsection{Bernstein components and types} \
\label{par:Bernstein}

Let Rep$(G)$ be the category of smooth representations of $G$ on complex vector spaces, and let
Irr$(G)$ be the set of (isomorphism classes of) irreducible objects in Rep$(G)$. We denote the
subset of supercuspidal irreducible representations by $\Irr_\cusp (G)$. Recall from \cite{BeDe} 
that every $\pi \in \Irr (G)$ has a cuspidal support $\mb{Sc}(\pi)$, which is unique up to
$G$-conjugation and consists of a Levi subgroup of $G$ and a supercuspidal irreducible
representation thereof.

For a Levi subgroup $L \subset G$ and $\pi_L \in \Irr_\cusp (L)$ we write
$[L,\pi_L ]_L = \mf s_L$ and
\[
\Irr (L)_{\mf s_L} = X_\nr (L) \pi_L = \{ \chi \otimes \pi_L : \chi \in X_\nr (L) \} ,
\]
an inertial class of supercuspidal $L$-representations. The $G$-orbit $[L,\pi_L ]_G = \mf s$
of $\mf s_L$ is by definition an inertial equivalence class for $G$. Notice that the group 
$N_G (L)$ acts naturally on $\Irr_\cusp (L)$, with $L$ acting trivially. To $\mf s$ (and the 
choice of a representative $\mf s_L$) one associates a finite group $W_{\mf s}$, the stabilizer
of $\mf s_L$ in $N_G (L) / L$. 

We denote the collection of inertial equivalence classes for $G$ by $\mf{Be}(G)$. Every 
$\mf s \in \mf{Be}(G)$ determines a Bernstein component of $\Irr (G)$:
\[
\Irr (G)_{\mf s} := \{ \pi \in \Irr (G) : \mb{Sc}(\pi) \in \mf s \} . 
\]
The associated Bernstein block Rep$(G)_{\mf s}$ is a direct factor of Rep$(G)$. 
The theory of the Bernstein centre \cite{BeDe} tells us that the centre of Rep$(G)_{\mf s}$
is $\mc O (\Irr (L)_{\mf s_L} / W_{\mf s})$.

As before we pick a facet $\mf f$ in the closure of $C_0$.
We assume that $\overline P_{\mf f}$ has a cuspidal unipotent representation $\overline \sigma$.
(This is a rather strong condition on the facet $\mf f$.) The inflation of $\overline \sigma$ to
$P_{\mf f}$ will be denoted $\sigma$, and its underlying vector space $V_\sigma$. It was shown in
\cite[\S 6]{MoPr2} and \cite[Theorem 4.8]{Mor2} that $(P_{\mf f},\sigma)$ is a type for $G$.
This has the following consequences \cite[Theorem 4.3]{BuKu}:
\begin{itemize}
\item Let Rep$(G)_{(P_{\mf f},\sigma)}$ be the category of smooth $G$-representations that are 
generated by their $\sigma$-isotypical vectors. Then Rep$(G)_{(P_{\mf f},\sigma)}$ is a direct
factor of Rep$(G)$, a direct sum of finitely Bernstein blocks.
\item Let $\mc H (G,P_{\mf f},\sigma)$ be the $G$-endomorphism algebra of the module
$\mr{ind}_{P_{\mf f}}^G (\sigma,V_\sigma)$. Then
\begin{equation}\label{eq:3.1}
\begin{array}{ccc}
\Rep (G)_{(P_{\mf f},\sigma)} & \to & \mr{Mod}(\mc H (G,P_{\mf f},\sigma)) \\
V & \mapsto & \mr{Hom}_{P_{\mf f}} (V_\sigma,V)
\end{array}
\end{equation}
is an equivalence of categories.
\end{itemize}
If $(P_{\mf f'}, \sigma')$ are data of the same kind as $(P_{\mf f},\sigma)$, then by 
\cite[Theorem 5.2]{MoPr2} the two associated subcategories of $\Rep (G)$ are either 
disjoint or equal. Moreover, by \cite[1.6.b]{LusUni1} 
\begin{multline}\label{eq:3.30}
\Rep (G)_{(P_{\mf f},\sigma)} = \Rep (G)_{(P_{\mf f'},\sigma')} 
\qquad \text{if and only if}\\
\text{there exists a } g \in G \text{ such that } P_{\mf f'} = g P_{\mf f} g^{-1}
\text{ and } \sigma' \cong g \cdot \sigma .
\end{multline}
The category of unipotent $G$-representations is defined as the full subcategory of
Rep$(G)$ generated by the $\Rep (G)_{(P_{\mf f},\sigma)}$ as above. By \eqref{eq:3.30}
\[
\Rep_\unip (G) = \prod\nolimits_{(P_{\mf f},\sigma) / G\text{-conjugation}} 
\Rep (G)_{(P_{\mf f},\sigma)} .
\]
We want to make the structure of $\mc H (G,P_{\mf f},\sigma)$ more explicit. This will
involve extending $\sigma$ to a representation of $\hat P_{\mf f}$ and analysing the
Hecke algebra for that type. Up to twists by $X_\nr (L_{\mf f})$, there
are only finitely many ways to extend $\sigma |_{P_{L,\mf f}}$ to a representation of
$N_{L_{\mf f}}(P_{L,\mf f})$, say $\hat \sigma_i \; (i = 1, \ldots, e_\sigma)$. Then
\[
\mf s_i = [L_{\mf f}, \mr{ind}_{N_{L_{\mf f}}(P_{L,\mf f})}^{L_{\mf f}} (\hat \sigma_i) ]_G
\] 
is an inertial equivalence class for $\Irr (G)$ and by \cite[Theorem 4.3]{Mor2}:
\begin{equation}\label{eq:3.15}
\begin{array}{ccc}
\Rep (G)_{(P_{\mf f},\sigma)} & = & \prod\nolimits_{i=1}^{e_\sigma} \Rep (G)_{\mf s_i} , \\
\Irr (G)_{(P_{\mf f},\sigma)} & = & \bigsqcup\nolimits_{i=1}^{e_\sigma} \Irr (G)_{\mf s_i} . 
\end{array}
\end{equation}
By \cite[\S 2.6]{BuKu}, $\mc H (G, P_{\mf f},\sigma)$ is naturally isomorphic to 
\[
\big\{ f \in C_c (G, \mathrm{End}_\C (V_\sigma^*)) : f(p_1 g p_2) =
\sigma^\vee (p_1) f(g) \sigma^\vee (p_2) \; \forall g \in G, p_1 ,p_2 \in P_{\mf f} \big\} ,
\]
where $(\sigma^\vee, V_\sigma^*)$ denotes the contragredient of $\sigma$. According to
\cite[\S 3.1]{Mor2} and \cite[1.18]{LusUni1} the support of $\mc H (G,P_{\mf f},\sigma)$ in
$G$ is $P_{\mf f} N_W (W_J) P_{\mf f}$. This makes sense because $N_G (S) \cap P_{C_0} =
Z_G (S) \cap P_{\mf f}$ is contained in $P_{\mf f}$. Moreover the group $W(J,\sigma)$
indexes a $\C$-basis $\{ T_w \}$ of $\mc H (G,P_{\mf f},\sigma)$, such that
$T_w$ has support $P_{\mf f} w P_{\mf f}$. We note that this is a little easier than in
\cite{Mor2} -- the crucial difference is that the cuspidal unipotent representation $\sigma$
is stable under automorphisms of $P_{\mf f}$, so the entire group $N_W (W_J) /W_J$ 
stabilizes $\sigma$.

The subgroup $N_G (P_{\mf f})$ supports a subalgebra of $\mc H (G,P_{\mf f},\sigma)$,
which by \cite[\S 1.19]{LusUni1} is isomorphic to the group algebra
\begin{equation}\label{eq:3.3}
\C [\Omega_{\mf f} ] = \C[N_G (P_{\mf f}) / P_{\mf f} ].
\end{equation}
The construction of the isomorphism involves the choice of an extension of $\sigma$ to 
a representation of $N_G (P_{\mf f})$.

When $P_{\mf f}$ is a maximal parahoric subgroup of $G$, \eqref{eq:3.3} coincides with
$\mc H (G,P_{\mf f},\sigma)$. 
The subalgebra $\mc H_\af (G,P_{\mf f},\sigma)$ spanned by $\{ T_w : w \in 
W_\af (J,\sigma)\}$ (i.e. supported on $P_{\mf f} W_\af (J,\sigma) P_{\mf f}$) 
is isomorphic to the Iwahori--Hecke algebra of the Coxeter system $(W_\af (J,\sigma), 
S_{\mf f,\af})$ with parameters as in \cite[\S 1.18]{LusUni1}. Together
with \eqref{eq:3.3} that gives a description as an extended affine Hecke algebra:
\begin{equation}\label{eq:3.5}
\mc H (G,P_{\mf f},\sigma) = \mc H_\af (G,P_{\mf f},\sigma) \rtimes \Omega_{\mf f} .
\end{equation}
According to \cite[\S 1--2]{LusUni1}, $\mc H_\af (G,P_{\mf f},\sigma)$ 
is isomorpic to an affine Hecke algebra determined by:
\begin{itemize}
\item The lattice $X (J)$ and its dual $X (J)^\vee$.
\item The root system $R_{\mf f}$ in $X(J)$ from \cite[\S 2.22]{LusUni1} (denoted $\mc R$ 
over there). It is indexed by $\Delta_{\mf f}$ and has Weyl group $W^\circ (J,\sigma)$.
\item The dual root system $R_{\mf f}^\vee$ from \cite[\S 2.22]{LusUni1} 
(denoted $\mc R'$ over there). 
\item The set of affine reflections $S_{\mf f,\af} = \{ s_i : i \in \Delta_\af 
\setminus J\}$ with parameter function $q_K^{\mc N}$ as in \cite[\S 1.18]{LusUni1}. 
\end{itemize}
For a character $\psi$ of $\Omega_{\mf f,\tor}$, let $e_\psi \in \C [\Omega_{\mf f}]$
be the associated idempotent. We can decompose \eqref{eq:3.5} as in 
\cite[\S 1.20]{LusUni1}:
\begin{equation}\label{eq:3.12}
\mc H (G,P_{\mf f},\sigma) = \bigoplus\nolimits_{\psi \in \Irr (\Omega_{\mf f,\tor})}
\mc H_\af (G,P_{\mf f},\sigma) \rtimes e_\psi \C[\Omega_{\mf f}] .
\end{equation}
By \eqref{eq:3.14} $\psi$ can also be regarded as a character of $\hat P_{\mf f} / P_{\mf f}$. 

Let $\hat \sigma$ be an extension of $\sigma$ to an irreducible representation of 
$\hat P_{\mf f}$, as in \cite[\S 1.16]{LusUni1}. We may assume that it comes from the 
extension of $\sigma$ used in \cite[\S 1.19]{LusUni1} to construct an isomorphism with 
\eqref{eq:3.3}.
The other extensions of $\sigma$ to $\hat P_{\mf f}$ are $\hat \sigma \otimes \psi$
with $\psi \in \Irr (\hat P_{\mf f} / P_{\mf f}) = \Irr (\Omega_{\mf f,\tor})$.
By \eqref{eq:3.30} different extensions of $\sigma$ to $\hat P_{\mf f}$ cannot be
conjugate by elements of $G$. Comparing with \eqref{eq:3.15} and taking \eqref{eq:3.13} 
into account, we see that there is a unique $i$ such that 
\[
(\hat \sigma \otimes \psi) \big|_{\hat{P}_{L,\mf f}} = 
\hat \sigma_i \big|_{\hat{P}_{L,\mf f}},
\]
In this situation we henceforth write
\[
\mf s_\psi = \mf s_i =  [L_{\mf f}, \mr{ind}_{N_{L_{\mf f}}(P_{L,\mf f})}^{L_{\mf f}} 
(\hat \sigma_i) ]_G ,
\]
and we denote the underlying inertial equivalence class for $L_{\mf f}$ by $\mf s_{L,\psi}$.

\begin{thm}\label{thm:3.4}
\enuma{
\item $(\hat P_{\mf f},\hat \sigma \otimes \psi)$ is a type for the single Bernstein block 
$\Rep (G)_{\mf s_\psi}$. 
\item There is an equivalence of categories
\[
\begin{array}{ccc}
\Rep (G)_{\mf s_\psi} & \to & \mr{Mod}(\mc H (G,\hat P_{\mf f},\hat \sigma \otimes \psi)) \\
V & \mapsto & \mr{Hom}_{\hat P_{\mf f}} (\hat \sigma \otimes \psi ,V)
\end{array}.
\]
\item Part (b) restricts to an equivalence between the respective subcategories of
tempered representations.
\item Suppose that $V \in \Rep (G)_{\mf s_\psi}$ has finite length. Then $V$ is essentially
square-integrable if and only if $\mr{Hom}_{\hat P_{\mf f}} (\hat \sigma \otimes \psi ,V)$
is an essentially discrete series \\
$\mc H (G,\hat P_{\mf f},\hat \sigma \otimes \psi)$-module.
}
\end{thm}
\begin{proof}
(a) is a special case of \cite[Theorem 4.7]{Mor2}.\\
(b) is a consequence of part (a) and \cite[Theorem 4.3]{BuKu}.\\
(c) For finite length tempered representations see \cite[Theorem 10.1]{DeOp}.
For general tempered representations see \cite[Theorem 3.12 and p.42]{SolComp}.\\
(d) In \eqref{eq:3.27} we saw that rk$(R_{\mf f}) = \mr{rk}(\Phi (G,L_{\mf f}))$.
It follows that under the isomorphism of $\mc H_\af (G,P_{\mf f},\sigma)$ with an
affine Hecke algebra, as described after \eqref{eq:3.5}, $R_{\mf f}$ corresponds to
a full rank root subsystem of $\Phi (G,L_{\mf f})$. This allows us to apply 
\cite[Theorem 3.9.b]{SolComp}, which proves the claim.
\end{proof}

From \eqref{eq:3.30} and Theorem \ref{thm:3.4} we get an equivalence of categories
\begin{equation}\label{eq:3.32}
\begin{array}{ccc}
\Rep_\unip (G) & \to & \mr{Mod}\Big( \bigoplus_{(\mf f,\hat \sigma \otimes \psi) / G} 
\mc H (G,\hat P_{\mf f},\hat \sigma \otimes \psi) \Big) \\
V & \mapsto & \bigoplus_{(\mf f,\hat \sigma \otimes \psi) / G}
\mr{Hom}_{\hat P_{\mf f}} (\hat \sigma \otimes \psi ,V)
\end{array}. 
\end{equation}
By \cite[\S 1.20]{LusUni1} the Hecke algebra of $(\hat P_{\mf f},\hat \sigma \otimes \psi)$ 
can be written as
\begin{equation}\label{eq:3.16}
\mc H (G,\hat P_{\mf f}, \hat \sigma \otimes \psi) =  
\mc H_\af (G,P_{\mf f},\sigma) \rtimes e_\psi \C[\Omega_{\mf f}] \cong
\mc H_\af (G,P_{\mf f},\sigma) \rtimes \Omega_{\mf f} / \Omega_{\mf f,\tor} . 
\end{equation}
By Proposition \ref{prop:3.2}, each algebra \eqref{eq:3.16} is isomorphic to an affine Hecke 
algebra associated to the almost the same data as $\mc H_\af (G,P_{\mf f},\sigma)$ above.
Only the lattices are different, namely, $\mc H (G,\hat P_{\mf f}, \hat \sigma \otimes \psi)$
comes from the lattice $X_{\mf f} \subset \mh A_{\mf f}$ and its dual $X_{\mf f}^\vee$.

Let us check that these contain all the appropriate (co)roots. From the definition 
\eqref{eq:3.29} we see that $X_{\mf f}$ contains $X(J)$, so it also contains $R_{\mf f}$. 
The constructions in \cite[\S 2]{LusUni1} involve the coroot lattice $\Z R_{\mf f}^\vee$, 
which is shown to coincide with a certain lattice $\mc L'$. The dual lattice 
$\mc L \subset \mh A_{\mf f}$ contains $X_{\mf f}$, with equality if $G$ is adjoint.
Hence $R_{\mf f}^\vee \subset \mc L' \subset X_{\mf f}^\vee$.

\subsection{Relations with the cuspidal and the adjoint cases} \

We want to relate the torus $\Irr (X_{\mf f})$ and the Weyl group $W^\circ (J,\sigma)$
associated to $\mc H (G,\hat P_{\mf f}, \hat \sigma \otimes \psi)$ with the torus
$\Irr (L_{\mf f})_{\mf s_{L,\psi}}$ and the finite group $W_{\mf s_\psi}$ associated by 
Bernstein to $\Rep (G)_{\mf s_i}$. Let $X_\nr (L_{\mf f},\sigma)$ be the stabilizer of
$\mr{ind}_{N_{L_{\mf f}}(P_{L,\mf f})}^{L_{\mf f}} (\hat \sigma_i) \in \Irr_\cusp (L_{\mf f})$
in $X_\nr (L_{\mf f})$, with respect to the action of tensoring by unramified characters.

The next result stems largely from \cite{Mor2}.

\begin{thm}\label{thm:3.1}
\enuma{
\item Let $Q$ be a Levi subgroup of $G$ containing $L_{\mf f}$ and write 
$\hat P_{Q,\mf f} = \hat P_{\mf f} \cap Q$. Then
$(\hat P_{\mf f}, \hat \sigma \otimes \psi)$ is a cover of 
$\big( \hat P_{Q,\mf f}, (\hat \sigma \otimes \psi) |_{\hat P_{Q,\mf f}} \big) $.
\item $\mc H \big( L_{\mf f}, \hat P_{L,\mf f}, (\hat \sigma \otimes \psi) 
|_{\hat P_{L,\mf f}} \big) \cong \C [X_{\mf f}]$.
\item There are homeomorphisms of complex tori
\[
X_\nr (L_{\mf f}) / X_\nr (L_{\mf f},\sigma) \to \Irr (L_{\mf f})_{\mf s_{L,\psi}}
\to \Irr (X_{\mf f}) ,
\]
such that the composed map (between the outer terms) is a natural group homomorphism.
} 
\end{thm}
\begin{proof}
(a) For $Q = L_{\mf f}$ this is \cite[Corollary 3.10]{Mor2}. Via the transitivity of
covers \cite[Proposition 8.5]{BuKu} that implies it for other $Q$.\\
(b) We apply the earlier results from this section with $L_{\mf f}$ instead of $G$.
In $\mc B (\mc L_{\mf f},K)$, $\mf f$ becomes a minimal facet and $P_{L,\mf f} =
P_{\mf f} \cap L_{\mf f}$ is a maximal parahoric subgroup of $L_{\mf f}$. Hence
$W_{L,\af} (J,\sigma |_{P_{L,\mf f}} ) =  1$ and
\[
\mc H \big( L_{\mf f}, \hat P_{L,\mf f}, (\hat \sigma \otimes \psi) 
|_{\hat P_{L,\mf f}} \big) \cong e_\psi \C [\Omega_{L,\mf f}] \cong
\C [\Omega_{L,\mf f} / \Omega_{L,\mf f,\tor}] .
\]
We still need to identify the subgroup $\Omega_{L,\mf f}$ of $W(J,\sigma)$. It
consists of all elements of $W(J,\sigma)$ that have a representative in $N_{L_{\mf f}}(S)$.
By \eqref{eq:3.13} and \eqref{eq:3.14}, $\Omega_{L,\mf f}$ contains the group
$\Omega_{\mf f,\tor} \cong \hat P_{L,\mf f} / P_{L,\mf f}$, namely as $\Omega_{L,\mf f,\tor}$,
(the  group $\Omega_{\mf f,\tor}$ with respect to $L_{\mf f}$). From 
Proposition \ref{prop:3.2} we see that $X_{\mf f} \subset X_* (Z_G (S))$ can be represented 
in $Z_G (S) = Z_{L_{\mf f}}(S)$, so $X_{\mf f} \subset \Omega_{L,\mf f} / \Omega_{\mf f,\tor}$. 
Proposition \ref{prop:3.2} (for $L_{\mf f}$) shows that 
$\Omega_{L,\mf f} / \Omega_{\mf f,\tor}$ acts by translations on $\mh A_{\mf f}$ 
(which we now view as the affine building of $Z(L_{\mf f})^\circ$).
By Proposition \ref{prop:3.2} the only elements of $W(J,\sigma) / \Omega_{\mf f,\tor}$ 
which act on $\mh A_{\mf f}$ by translations, are those of $X_{\mf f}$. Hence
$\Omega_{L,\mf f} / \Omega_{\mf f,\tor} = X_{\mf f}$. 

From the above, \eqref{eq:3.4}, \eqref{eq:3.7} and Proposition \ref{prop:3.2} we see that 
$\Omega_{L,\mf f} = X(J)$.\\
(c) The equivalence of categories from Theorem \ref{thm:3.4}.b for $L_{\mf f}$ 
restricts to a bijection
\begin{equation}\label{eq:3.19}
\Irr (L_{\mf f})_{\mf s_{L,\psi}} \to \mr{Irr}(\mc H (L_{\mf f},\hat P_{L,\mf f},
(\hat \sigma \otimes \psi)|_{\hat P_{L,\mf f}} )) :
V \mapsto \mr{Hom}_{\hat P_{L,\mf f}}(\hat \sigma \otimes \psi, V) .
\end{equation}
The left hand side is 
\[
\big\{ \chi \otimes \mr{ind}_{N_{L_{\mf f}}(P_{L,\mf f})}^{L_{\mf f}} (\hat \sigma_i) :
\chi \in X_\nr (L_{\mf f} \big\} ,
\]
which by construction admits a simply transitive action of $X_\nr (L_{\mf f}) / 
X_\nr (L_{\mf f},\sigma)$. By part (a) the right hand side of \eqref{eq:3.19} can be 
identified with $\Irr (\C[X_{\mf f}]) = \Irr (X_{\mf f})$. Let $i(\sigma)_\psi$ be
the unique unramified twist of $\mr{ind}_{N_{L_{\mf f}}(P_{L,\mf f})}^{L_{\mf f}} 
(\hat \sigma_i)$ which under \eqref{eq:3.19} maps to the trivial representation of 
$X_{\mf f}$. We take $\chi \mapsto \chi \otimes i(\sigma)_\psi$ as map 
$X_\nr (L_{\mf f}) / X_\nr (L_{\mf f}, \sigma) \to \Irr (L_{\mf f})_{\mf s_{L,\psi}}$.
The canonical maps
\begin{equation}\label{eq:3.24}
X_{\mf f} \hookrightarrow Z_G (S) / Z_G (S)_\cpt \leftarrow Z_G (S) \hookrightarrow L_{\mf f}
\end{equation}
induce a natural group homomorphism 
\begin{equation}\label{eq:3.20}
X_\nr (L_{\mf f}) \to \Irr (X_{\mf f}) : \chi \mapsto \chi |_{X_{\mf f}}. 
\end{equation}
Then Theorem \ref{thm:3.4}.b and \eqref{eq:3.19} send $\chi \otimes i(\sigma)_\psi$ to
$\chi |_{X_{\mf f}}$. As \eqref{eq:3.19} is bijective, \eqref{eq:3.20} induces a group 
isomorphism $X_\nr (L_{\mf f}) / X_\nr (L_{\mf f}, \sigma) \to \Irr (X_{\mf f})$.
\end{proof}

\begin{lem}\label{lem:3.3} 
Let $i(\sigma)_\psi \in \Irr (L_{\mf f})_{\mf s_{L,\psi}}$ be the unique element that
maps to $\mr{triv}_{X_{\mf f}}$ under \eqref{eq:3.19}.
\enuma{
\item $i (\sigma)_\psi$ is fixed by $W_{\mf s_\psi}$. 
\item The canonical map $W \to W(G,S)$ induces an isomorphism 
$W^\circ (J,\sigma) \to W_{\mf s_\psi}$, and this makes the homeomorphisms in
Theorem \ref{thm:3.1}.c $W_{\mf s_\psi}$-equivariant.
} 
\end{lem}
\begin{proof}
(a) It is known from \cite[\S 3]{Lus-Gr} that
\[
Z\big( \mc H (G,\hat P_{\mf f}, \hat \sigma \otimes \psi ) \big) \cong 
\C[X_{\mf f}]^{W^\circ (J,\sigma)} = 
\mc O \big( \Irr (X_{\mf f}) / W^\circ (J,\sigma) \big).
\]
On the other hand, by \cite[Th\'eor\`eme 2.13]{BeDe}, the centre of the category
$\Rep (G)_{\mf s_\psi}$ is $\mc O (\Irr (L_{\mf f})_{\mf s_{L,\psi}} / W_{\mf s_\psi})$.
In view of Theorem \ref{thm:3.1}.a, we may use the properties of covers, in particular
\cite[Corollary 8.4]{BuKu}. The version of \cite[Corollary 8.4]{BuKu} with normalized 
parabolic induction \cite[Lemma 4.1]{SolComp} says that, for any parabolic subgroup $Q U_Q$ 
of $G$ whose Levi factor $Q$ contains $L_{\mf f}$, the following diagram commutes:
\begin{equation}\label{eq:3.36}
\begin{array}{ccc}
\Rep (G)_{\mf s_\psi} & \longrightarrow & 
\mr{Mod} \big( \mc H (G,\hat P_{\mf f}, \hat \sigma \otimes \psi ) \big) \\
\uparrow I^G_{Q U_Q} & & \uparrow 
\mr{ind}_{\mc H (Q,\hat P_{Q,\mf f}, \hat \sigma \otimes \psi |_{\hat P_{Q,\mf f}})}^{
\mc H (G,\hat P_{\mf f}, \hat \sigma \otimes \psi )} \\
\Rep (Q)_{\mf s_{Q,\psi}} & \longrightarrow & 
\mr{Mod}(\mc H (Q,\hat P_{Q,\mf f}, \hat \sigma \otimes \psi |_{\hat P_{Q,\mf f}})
\end{array} .
\end{equation}
Here the embedding 
\[
\mc H (Q,\hat P_{Q,\mf f}, \hat \sigma \otimes \psi |_{\hat P_{Q,\mf f}}) \to 
\mc H (G,\hat P_{\mf f}, \hat \sigma \otimes \psi )
\]
depends on the choice of a parabolic subgroup $Q U_Q$ of $G$ with Levi factor $Q$. 

By Theorem \ref{thm:3.1}.b we may take $\C [X_{\mf f}]$ as the Hecke algebra for 
$(\hat P_{L,\mf f},(\hat \sigma \otimes \psi)|_{\hat P_{L,\mf f}})$.
Then \eqref{eq:3.36} says in particular that Theorem \ref{thm:3.4}.b sends 
\begin{equation}\label{eq:3.21}
I^G_{L_{\mf f} U_{L_{\mf f}}} (\chi \otimes i(\sigma)_\psi) \quad \text{to} \quad 
\mr{ind}_{\C [X_{\mf f}]}^{\mc H (G,\hat P_{\mf f}, 
\hat \sigma \otimes \psi )} (\chi |_{X_{\mf f}}) .
\end{equation}
On both sides of Theorem \ref{thm:3.4}.b the centres of the categories can be
detected by their actions on parabolically induced representations as in \eqref{eq:3.21}.
Thus the bijection $\Irr (L_{\mf f})_{\mf s_{L,\psi}} \to \Irr (X_{\mf f})$ from
Theorem \ref{thm:3.1}.c induces a bijection
\begin{equation}\label{eq:3.22}
\Irr (L_{\mf f})_{\mf s_{L,\psi}} / W_{\mf s_\psi} \to \Irr (X_{\mf f}) / W^\circ (J,\sigma) .
\end{equation}
As both $W_{\mf s_\psi}$ and $W^\circ (J,\sigma)$ are finite groups acting faithfully
by automorphisms of complex affine varieties, the subspaces of these tori on which the
isotropy groups are trivial form Zariski-open dense subvarieties. For $w \in W_{\mf s_\psi}$
and $t \in \Irr (L_{\mf f})_{\mf s_{L,\psi}}$ with image $t' \in \Irr (X_{\mf f})$ such that 
both have trivial stabilizer, the condition that $w(t)$ maps to $w' (t')$ completely determines 
$w'$. That yields a unique group isomorphism $W_{\mf s_\psi} \to W^\circ (J,\sigma)$
such that $\Irr (L_{\mf f})_{\mf s_{L,\psi}} \to \Irr (X_{\mf f})$ becomes 
$W_{\mf s_\psi}$-equivariant. 
Clearly the trivial representation of $X_{\mf f}$ is fixed by $W^\circ (J,\sigma)$, and
hence $i(\sigma)_\psi$ is fixed by $W_{\mf s_\psi}$. \\
(b) By the above, the bijection 
\[
X_\nr (L_{\mf f}) / X_\nr (L_{\mf f},\sigma) \to \Irr (L_{\mf f})_{\mf s_{L,\psi}} :
\chi \mapsto \chi \otimes i(\sigma)_\psi 
\]
from Theorem \ref{thm:3.1}.c is also $W_{\mf s_\psi}$-equivariant. Consequently the 
group isomorphism 
\begin{equation}\label{eq:3.25}
X_\nr (L_{\mf f}) / X_\nr (L_{\mf f}, \sigma) \to \Irr (X_{\mf f})
\end{equation}
induced by \eqref{eq:3.24} is $W_{\mf s_\psi}$-equivariant. The group $W^\circ (J,\sigma)$
acts naturally on $X_{\mf f}$, the action is induced by conjugation in $N_G (S)$.
Conjugation also yields actions of $N_G (S)$ and $W^\circ (J,\sigma)$ on $L_{\mf f}$
and on $X_\nr (L_{\mf f})$. Conjugation by $Z_G (S)$ does not change (unramified)
characters, so the action of $W^\circ (J,\sigma)$ factors through
\begin{equation}\label{eq:3.26}
W^\circ (J,\sigma) \to W \to W(G,S). 
\end{equation}
Thus \eqref{eq:3.25} is equivariant for the canonical actions of $W^\circ (J,\sigma)$,
and is equivariant with respect to the above isomorphism $W^\circ (J,\sigma) \cong
W_{\mf s_\psi}$. By the uniqueness of that isomorphism, it must agree with the map 
induced by \eqref{eq:3.26}.
\end{proof}

We will relate the objects in Subsection \ref{par:Bernstein} for $\mc G$ with those for 
its adjoint group $\mc G_\ad = \mc G / Z(\mc G)$. It is well-known that the enlarged 
Bruhat--Tits building depends only on $\mc G$ up to isogeny, and that the semisimple 
Bruhat--Tits buildings of $(\mc G,K)$ and $(\mc G_\ad,K)$ can be identified. Since 
parahoric subgroups correspond (bijectively) to facets of the semisimple Bruhat--Tits 
building, $\mc G (K)$ and $\mc G_\ad (K)$ have the same set of parahoric subgroups. 
(In $(\mc G_\ad (K)$ more of them can be conjugate, though.) 

The $\mf o_K$-group $\mc G_{\mf f}^\circ$ is isogenous to the direct product of 
$\mc G_{\ad,\mf f}^\circ$ and an $\mf o_K$-torus. Hence $\overline P_{\mf f}$ is 
isogenous (as $k$-group) to the direct product of $\overline P_{\mf f,\ad}$ and a 
$k$-torus. The collection of (cuspidal) unipotent representations of a connected 
reductive $k$-group $\mc H (k)$ only depends on $\mc H$ up to isogeny 
\cite[Proposition 3.15]{Lus-Che} and a $k$-torus has just one irreducible unipotent 
representation, namely the trivial representation. Therefore we may identify the 
collections of cuspidal unipotent representations of $\overline P_{\mf f}$ and 
$\overline P_{\mf f,\ad}$. The same goes for the collections of cuspidal unipotent 
representations of $P_{\mf f}$ and $P_{\mf f,\ad}$. We will denote the cuspidal 
unipotent representation of $P_{\mf f,\ad}$ corresponding to 
$\sigma \in \Irr (P_{\mf f})$ by $\sigma_\ad$.

\begin{lem}\label{lem:4.3}
The following objects are the same for $(G,P_{\mf f},\sigma)$ and for 
$(G_\ad, P_{\mf f,\ad}, \sigma_\ad)$:

$S_{\mf f,\af}, W_\af (J,\sigma), W^\circ (J,\sigma), R_{\mf f}, R_{\mf f}^\vee$ and
$\mc H_\af (G,P_{\mf f},\sigma)$.
\end{lem}
\begin{proof}
The set $\Delta_\af$ depends only on $\mc{BT}(\mc G,K) = \mc{BT}(\mc G_\ad ,K)$ and $J$ is
determined by the facet $\mf f$, so the claim holds for $\Delta_\af \setminus J$ and for
$S_{\mf f,\af}$. Hence also for the Coxeter group $W_\af (J,\sigma)$ with generators
$S_{\mf f,\af}$. We can choose $S_{\mf f} \subset S_{\mf f,\af}$ in the same way for $G$ and 
for $G_\ad$, so the Weyl group $W^\circ (J,\sigma)$ generated by $S_{\mf f}$ does not change
under passage to the adjoint case. 

The construction of $R_{\mf f}$ and $R_{\mf f}^\vee$ in \cite[\S 2]{LusUni1} depends only on
$(W_\af (J,\sigma), S_{\mf f,\af})$, so it is the same for $(G,P_{\mf f})$ and for
$(G_\ad , P_{\mf f,\ad})$. The parameters $q_K^{\mc N (\alpha)}$ for $S_{\mf f,\af}$ 
used in $\mc H_\af (G,P_{\mf f},\sigma)$ are defined in \cite[\S 1.18]{LusUni1} and
\cite[\S 6.9 and \S 7.1]{Mor1}. For the parameter of $s_\alpha = s_i$, consider the 
standard parahoric subgroup $P_{J \cup \{i\}}$ of $G$ determined by $J \cup \{i\}$.
It contains $P_{\mf f} = P_J$, and $\mr{ind}_{P_J}^{P_{J \cup \{i\}}} (\sigma)$ is
a direct sum of two irreducible representations, say $\sigma_1$ and $\sigma_2$.
Write dim$(\sigma_j) = q_K^{n_j}$ with $n_j \in \Z_{\geq 0}$, then the parameter of
$s_i$ is 
\[
q_i = q_K^{|n_1 - n_2|} .
\]
It follows from \cite[Proposition 2.6]{Lus-Che} that the class of unipotent
representations of connected reductive groups over finite fields is closed under
parabolic induction. In particular $\mr{ind}_{P_J}^{P_{J \cup \{i\}}} (\sigma)$
is again unipotent, and independent of isogenies of the involved group. More
explicitly, via the map $P_{J \cup \{i\}} \to P_{J \cup \{i\},\ad}$ this representation 
is isomorphic to $\mr{ind}_{P_{J,\ad}}^{P_{J \cup \{i\},\ad}} (\sigma_\ad)$. (Of course 
the isomorphism can also be seen more elementarily.) It follows that pullback from 
$P_{J \cup \{i\},\ad}$ to $P_{J \cup \{i\}}$ also defines isomorphisms
$\sigma_{\ad,j} \cong \sigma_j$. Comparing their dimensions, we find that $q_{\ad,i} = q_i$.

The Iwahori--Hecke algebra $\mc H_\af (G,P_{\mf f},\sigma)$ depends only on
$W_\af (J,\sigma), S_\af$ and the parameters $q_i$ for $i \in \Delta_\af \setminus J$,
so it is naturally isomorphic to $\mc H_\af (G_\ad,P_{\mf f,\ad},\sigma_\ad)$.
\end{proof}

Now we can formulate a precise comparison between the affine Hecke algebras 
\begin{equation}\label{eq:4.6}
\mc H (G,\hat P_{\mf f},\hat \sigma \otimes \psi) \quad \text{and} \quad
\mc H (G_\ad,\hat P_{\mf f,\ad}, \hat \sigma_\ad \otimes \psi_\ad),
\end{equation}
for any choice of $\psi_\ad$. From Lemma \ref{lem:4.3}
and \eqref{eq:3.16} we see that they only differ in the underlying lattices:
$X_{\mf f}$ is usually not equal to the weight lattice $X_{\mf f, \ad}$. 

\section{Comparison of Hecke algebras}

Let $\mc G$ be a connected reductive $K$-group which splits over an unramified extension
of $K$. We denote the set of ($G^\vee$-equivalence classes of) unramified L-parameters
for $G$ by $\Phi_\nr (G)$. We indicate the set of unipotent representations in 
$\Irr (G)$ (or Rep$(G)$ etc.) by a subscript ``unip". 

A character of $G$ is called weakly unramified if it is trivial on the kernel of 
the Kottwitz homomorphism $G \to \Omega$. The group $X_\Wr (G)$ of weakly unramified
characters $G \to \C^\times$ is naturally isomorphic to an object coming from ${}^L G$:
\begin{equation}\label{eq:4.15}
X_\Wr (G) \cong (Z (G^\vee)^{\mb I_K} )_\Fr \quad \subset H^1 (\mb W_K ,Z(G^\vee)) , 
\end{equation}
see \cite[\S 3.3.1]{Hai}. Its identity component is the group $X_\nr (G)$ of unramified 
characters $G \to \C^\times$. Via \eqref{eq:4.15} and \eqref{eq:2.10}, $X_\Wr (G)$ acts 
naturally on $\Phi_e (G)$, while it acts on $\Rep (G)$ by tensoring.

Recall that the HII conjectures \cite{HII} compare the formal degree of a square-integrable
modulo centre $G$-representation with (the specialization at $s=0$ of) the adjoint 
$\gamma$-factor of its L-parameter.

We formulate the main result of \cite{FOS}, and then derive some useful consequences.

\begin{thm}\label{thm:4.1} \textup{\cite[Theorems 2 and 3]{FOS}} \\
There exists a bijective map
\[
\Irr_{\cusp,\unip}(G) \to \Phi_{\nr,\cusp}(G) : \pi \mapsto (\lambda_\pi ,\rho_\pi)
\]
with the following properties:
\begin{enumerate}[(i)]
\item Equivariance with respect to the natural actions of $X_\Wr (G)$.
\item Compatibility with almost direct products of reductive groups.
\item Equivariance with respect to $\mb W_K$-automorphisms of the absolute
root datum of $\mc G$.
\item The map $\pi \mapsto \lambda_\pi$ makes the HII conjectures true for 
$\Irr_{\cusp,\unip}(G)$.
\item The map $\pi \mapsto \lambda_\pi$ is determined by the above properties
(i), (ii) and (iv), up to twisting by weakly unramified characters.
\end{enumerate}
\end{thm}

Let $\mf{Be}(G)_\unip$ be the subset of $\mf{Be}(G)$ obtained from $\Irr_\unip (G)$,
and similarly let $\mf{Be}^\vee (G)_\nr$ be the subset of $\mf{Be}^\vee (G)$
obtained from $\Phi_{\nr,e}(G)$.

\begin{prop}\label{prop:4.2}
\enuma{
\item Theorem \ref{thm:4.1} induces a bijection 
\[
\mf{Be}(G)_\unip \to \mf{Be}^\vee (G)_\nr : \mf s \mapsto \mf s^\vee .
\]
If $\mf s$ can be represented by a cuspidal inertial class for a Levi subgroup $L$ of $G$,
then so can $\mf s^\vee$, and conversely.
\item Suppose that $\mf s = [L,\pi_L]_G$ for some $\pi_L \in \Irr_{\cusp,\unip}(L)$.
There is a natural isomorphism $W_{\mf s} \cong W_{\mf s^\vee}$, and it makes the bijection
$\Irr (L)_{\mf s_L} \to \Phi_e (L)^{\mf s_L^\vee}$ $W_{\mf s}$-equivariant.
}
\end{prop}
\begin{proof}
(a) By Corollary \ref{cor:1.3} it suffices to show this for inertial equivalence classes
based on objects for a Levi subgroup $L$ of $G$. By property (i), the bijection in Theorem 
\ref{thm:4.1} induces a bijection
\[
\Irr_{\cusp,\unip}(G) / X_\nr (G) \to \Phi_{\nr,\cusp}(G) / X_\nr (G) .
\] 
Applying this to $L$, we obtain a bijection between the cuspidal inertial classes for
$\Irr_\unip (L)$ and $\Phi_{\nr,e}(L)$, say
\begin{equation}\label{eq:4.1}
\mf s_L \longleftrightarrow \mf s_L^\vee .
\end{equation}
Two such classes $\mf s_L, \mf s'_L$ become the same in $\mf{Be}(G)$ if and only if they
are conjugate by an element of $N_G (L)$. As $L$ acts trivially here, this is equivalent to 
$\mf s_L$ and $\mf s'_L$ being in the same orbit under $N_G (L) / L$. The action of
$N_G (L) / L$ on $\Irr (L)$ comes from its action (by $\mb W_K$-equivariant automorphisms)
on the absolute root datum of $\mc L$.

Similarly (cf. \cite[(117)]{AMS1}), the two classes $\mf s_L^\vee, {\mf s'}_L^\vee$ 
become the same in $\mf{Be}^\vee (G)$ if and only if they are conjugate by an element of
$N_{G^\vee}(L^\vee \rtimes \mb W_K)$, or equivalently by an element of $N_{G^\vee}(L^\vee 
\rtimes \mb W_K) / L^\vee$. This action of $N_{G^\vee}(L^\vee \rtimes \mb W_K) / L^\vee$
is determined by its action on the absolute root datum of $L^\vee$ (or equivalently that
of $\mc L$).

By \cite[Proposition 3.1]{ABPSLLC} there is a natural isomorphism 
\begin{equation}\label{eq:4.2}
N_G (L) / L \cong N_{G^\vee}(L^\vee \rtimes \mb W_K ) / L^\vee .
\end{equation}
Its construction entails that both sides act in the same way on the absolute root datum of
$\mc L$. Now property (iii) of Theorem \ref{thm:4.1} for $L$ says that 
$\pi_L \mapsto (\phi_{\pi_L},\rho_{\pi_L})$ is equivariant for the indicated actions of 
\eqref{eq:4.2}. \\
(b) By definition $W_{\mf s}$ is the stabilizer of $\mf s_L \cong \Irr (L)_{\mf s_L}$ in
$N_G (L) / L$, and $W_{\mf s^\vee}$ is the stabilizer of $\mf s_L^\vee \cong 
\Phi_e (L)^{\mf s_L^\vee}$ in $N_{G^\vee}(L^\vee \rtimes \mb W_K) / L^\vee$. 
By the above $N_G (L) / L$-equivariance of \eqref{eq:4.1}, the isomorphism \eqref{eq:4.2} 
restricts to $W_{\mf s^\vee} \cong W_{\mf s}$. In particular the bijection 
$\Irr (L)_{\mf s_L} \to \Phi_e (L)^{\mf s_L}$ from Theorem \ref{thm:4.1}
becomes equivariant for $W_{\mf s}$.
\end{proof}

Let us compare the Hecke algebras for L-parameters to those in the adjoint case. 
Replacing $\mc G$ by $\mc G_\ad$ means that $G^\vee$ is replaced by ${G^\vee}_\Sc$, the 
simply connected cover of the derived group of $G^\vee$. Let $\mc L \subset \mc G$ be a 
Levi $K$-subgroup and write $\mc L_c = \mc L / Z(\mc G) \subset \mc G_\ad$, so that
$L_c^\vee \subset {G^\vee}_\Sc$ is the Levi subgroup determined by $L \subset G$. 

Let $(\phi_L ,\rho_L) \in \Phi_{\nr,\cusp}(L)$. Since $\phi_L (\Fr)$ determines 
$\phi_L |_{\mb W_F}$ completely, it is easy to lift $\phi_L$ to a L-parameter 
$\phi_{L,\ad}$ for $L_c$: we only have to pick a lift of $\phi_L (\Fr)$ in 
$L_c^\vee \rtimes \mb W_K$. Then 
\[
Z_{{L^\vee}_\Sc}( \phi_{L,\ad})^\circ = Z_{{L^\vee}_\Sc}( \phi_L)^\circ 
\]
(in fact these groups are trivial because $\phi_L$ is discrete) and
\begin{equation}\label{eq:4.13}
Z^1_{{L^\vee}_\Sc}( \phi_{L,\ad}) \subset Z^1_{{L^\vee}_\Sc}( \phi_L) .
\end{equation}
Hence $\mc S_{\phi_{L,\ad}}$ is naturally embedded in $\mc S_{\phi_L}$. Let $\rho_{L,\ad}$
be an irreducible representation of $\mc S_{\phi_{L,\ad}}$ appearing in
$\rho \big|_{\mc S_{\phi_{L,\ad}}}$. The conditions for $\rho_L$ to be cuspidal and
L-relevant depend only on 
\begin{equation}\label{eq:4.4}
(L^\vee_{\phi_{L,\ad}})^\circ = Z_{{L^\vee}_\Sc}( \phi_{L,\ad} (\mb W_K))^\circ = 
Z_{{L^\vee}_\Sc}( \phi_L (\mb W_K))^\circ = (L^\vee_{\phi_L})^\circ ,
\end{equation}
so $(\phi_{L,\ad},\rho_{L,\ad}) \in \Phi_{\nr,\cusp}(L_c)$.

Let $\pi_L \in \Irr_{\cusp,\unip}(L)$ and $\pi_{L,\ad} \in \Irr_{\cusp,\unip}(L_c)$ be the
representations associated to, respectively, $(\phi_L,\rho_L)$ and $(\phi_{L,\ad},\rho_{L,\ad})$
by Theorem \ref{thm:4.1}. The constructions in \cite[\S 14--15]{FOS} entail that, up to a
twist by a weakly unramified character, $\pi_L$ is contained in the pullback of $\pi_{L,\ad}$
along $q : L \to L_c$. 

Let $\mf s^\vee_{L,\ad}$ be the inertial class for $\Phi_e (L_c)$ containing 
$(\phi_{L,\ad},\rho_{L,\ad})$, and let $\mf s^\vee_\ad$ be the resulting inertial equivalence
class for $\Phi_e (G)$. We note that the canonical homomorphism $q^\vee : L_c^\vee \to L^\vee$ 
induces maps 
\begin{equation}\label{eq:4.14}
\mf s^\vee_{L,\ad} \to \mf s^\vee_L : (\phi'_{L,\ad},\rho_{L,\ad}) \mapsto
({}^L q \circ \phi'_{L,\ad}, \rho_L)
\end{equation}
and $\mf s^\vee_\ad \to \mf s^\vee$. These maps depend on the choice of $\rho_{L,\ad}$
and $\rho_L$, but given $\mf s^\vee_{L,\ad}$ and $\mf s^\vee_L$, they are canonical.

Let $\mf s \in \mf{Be}(G)$ and $\mf s_\ad \in \mf{Be}(G_\ad)$ be the inertial equivalence 
classes obtained from $\mf s^\vee$ and $\mf s^\vee_\ad$ via Proposition \ref{prop:4.2}. 

\begin{lem}\label{lem:4.4}
The following objects are the same for $G^\vee$ and for ${G^\vee}_\Sc$, up to natural
isomorphisms: $W_{\mf s^\vee}, \Phi_{\mf s^\vee}, \lambda, \lambda^*$.

For any $\alpha \in \Phi_{\mf s^\vee}$ the function $\theta_{\alpha,\ad} \in 
\mc O (\mf s^\vee_{L,\ad})$ is the composition of $\theta_\alpha \in \mc O (\mf s^\vee_L)$
with the canonical map $\mf s^\vee_{L,\ad} \to \mf s^\vee_L$.
\end{lem}
\begin{proof}
The canonical maps ${G^\vee}_\Sc \to G^\vee, G \to G_\ad$ and \eqref{eq:4.2} combine to
a commutative diagram
\begin{equation}\label{eq:4.3}
\begin{array}{ccc}
N_{{G^\vee}_\Sc}(L_c^\vee \rtimes \mb W_K) / L_c^\vee & \cong & N_{G_\ad}(L_c) / L_c \\
\downarrow & & \downarrow \\
N_{G^\vee}(L^\vee \rtimes \mb W_K) / L^\vee & \cong & N_G (L) / L 
\end{array}
\end{equation}
It is easy to see that the vertical maps in \eqref{eq:4.3} are isomorphisms. The groups
$W_{\mf s^\vee}, W_{\mf s^\vee_\ad}, W_{\mf s}$ and $W_{\mf s_\ad}$ are contained in
the corners of this diagram, as the subgroups stabilizing, respectively, $\mf s^\vee,
\mf s^\vee_\ad, \mf s$ and $\mf s_\ad$. In Proposition \ref{prop:4.2}.b we showed that
the rows in \eqref{eq:4.3} restrict to isomorphisms $W_{\mf s^\vee_\ad} \cong W_{\mf s_\ad}$
and $W_{\mf s^\vee} \cong W_{\mf s}$. By Lemmas \ref{lem:4.3} and \ref{lem:3.3}.b
the right column of \eqref{eq:4.3} restricts to an isomorphism $W_{\mf s} \to W_{\mf s_\ad}$.
By the commutativity of the diagram, the left column restricts to an isomorphism
$W_{\mf s^\vee} \to W_{\mf s^\vee_\ad}$.

By \eqref{eq:4.4} and \cite[Lemma 3.10]{AMS3}, both $\Phi_{\mf s^\vee}$ and 
$\Phi_{\mf s^\vee_\ad}$ come from the same root system 
$\Phi \big( (G^\vee_{\phi_L})^\circ , Z(L_c^\vee)^{\mb W_F,\circ} \big)$. 
This implies that the canonical map 
\begin{equation}\label{eq:4.5}
T_{\mf s^\vee_\ad} \to T_{\mf s^\vee} \text{ provides a bijection }
\Phi_{\mf s_\ad^\vee} \to \Phi_{\mf s^\vee} .
\end{equation}
Recall that all these roots evaluate to 1 on the basepoints of $\mf s_L^\vee$ and 
$\mf s_{L,\ad}^\vee$ \cite[Proposition 3.9.b]{AMS3}. Hence the functions $\theta_\alpha$ and 
$\theta_{\alpha,\ad}$ they determine on, respectively, $\mf s_L^\vee$ and $\mf s_{L,\ad}^\vee$,
are related by composition with the canonical map from \eqref{eq:4.14}.

As described after \eqref{eq:2.3}, the label functions $\lambda$ and $\lambda^*$ depend only
on objects living in $Z_{{G^\vee}_\Sc} (t \phi_L (\mb W_K))^\circ$, for a few
$t \in (Z(L^\vee)^{\mb I_K})^\circ_\Fr$. From the proof of \cite[Lemma 3.12]{AMS3} one sees that
if $t_\ad \phi_{L,\ad} \in \Phi_\nr (L_c)$ is suitable to compute $\lambda_\ad (\alpha)$ and 
$\lambda_\ad^* (\alpha)$, then its image $t \phi_L$ in $\Phi_\nr (L)$ is suitable to compute
$\lambda (\alpha)$ and $\lambda^* (\alpha)$. Using these $t$'s, we see that $\lambda, \lambda^*$
and $\lambda_\ad, \lambda_\ad^*$ are given by the same formulas, namely those in the proof of 
\cite[Lemma 3.12]{AMS3}. Hence $\lambda_\ad = \lambda$ and $\lambda_\ad^* = \lambda^*$, with
the canonical bijection \eqref{eq:4.5} in mind.
\end{proof}

From Lemma \ref{lem:4.4} and the discussion following \eqref{eq:2.4} we see that the affine 
Hecke algebras
\begin{equation}\label{eq:4.7}
\mc H (\mf s^\vee, \vec{v}) \quad \text{and} \quad \mc H (\mf s^\vee_\ad, \vec{v})
\end{equation}
have almost the same presentation. Only the tori differ, and those are related via the 
map $\mf s^\vee_{L,\ad} \to \mf s^\vee_L$ from \eqref{eq:4.14}.

\begin{thm}\label{thm:4.5}
Let $(G, \hat P_{\mf f}, \hat \sigma \otimes \psi)$ and $\mf s_\psi \in \mf{Be}(G)$ be as in
Theorem \ref{thm:3.4}, and let $\mf s^\vee_\psi \in \mf{Be}^\vee (G)$ be the image of 
$\mf s^\vee_\psi$ under Proposition \ref{prop:4.2}.a. Theorem \ref{thm:4.1} and Proposition 
\ref{prop:4.2}.b induce a unique algebra isomorphism
\[
\mc H (\mf s^\vee_\psi, \vec{v}) \to \mc H (G, \hat P_{\mf f}, \hat \sigma \otimes \psi) .
\]
It comes from a canonical isomorphism between the underlying based root data. 
The requirement that (under the correspondence between the roots on both sides) 
$\vec{v}^\lambda$ and $\vec{v}^{\lambda^*}$ must agree with the parameter function $q_K^{\mc N}$ 
for $\mc H_\af (G, P_{\mf f}, \sigma)$,  determines $\vec{v}$.
\end{thm}

\begin{lem}\label{lem:4.6}
Theorem \ref{thm:4.5} holds if $\mc G$ is absolutely simple and adjoint. Here $\vec{v}$ is
a single number, namely $q_K^{1/2} = |k|^{1/2}$.
\end{lem}
\begin{proof}
We are in the setting of \cite{LusUni1,LusUni2}. Our affine Hecke algebra 
$\mc H (\mf s_\psi^\vee,\vec{v})$ can be identified with $\mb H (G,G_J,\mc C,\mc F)$ (from 
\cite[\S 5.17]{LusUni1}, when $\mc G$ is an inner form of a $K$-split group) or more generally 
with $\mb H (G\theta,G_J,\mc C, \mc F)$ from \cite[\S 8.2]{LusUni2}. 

To match Lusztig's notations with ours, we must take $G = G^\vee_{\phi_L} , G_J = L_c^\vee ,
\mc C$ the adjoint orbit of $\log (u_{\phi_L})$ and $\mc F$ the cuspidal local system on
$\mc C$ determined by $\rho_L$. Then the construction of $\mc H (\mf s^\vee_\psi, \vec{v})$
in \cite{AMS1,AMS2,AMS3} boils down to the relevant parts of \cite{LusUni1,LusUni2}.
(In fact this was a starting point of the work of Aubert--Moussaoui--Solleveld.)

In \cite[Theorem 6.3]{LusUni1} and \cite[Theorem 10.11]{LusUni2} Lusztig exhibited, in
particular, a matching between Bernstein components for $\Irr_\unip (G)$ and for
$\Phi_{\nr,e}(G)$. The bijection in Theorem \ref{thm:4.1} comes from \cite{FeOp} and agrees
with Lusztig's parametrization of supercuspidal unipotent representations. Hence Lusztig's
matching of Bernstein components is the same as in Proposition \ref{prop:4.2}.a.

As explained in the proofs of \cite[Theorem 6.3]{LusUni1} and \cite[Theorem 10.11]{LusUni2},
this matching is such that the corresponding affine Hecke algebras on both sides have the same
Iwahori--Matsumoto presentation. Let us make this more explicit. We can reformulate it by 
saying that $\mc H (G, \hat P_{\mf f}, \hat \sigma \otimes \psi)$
and $\mb H (G\theta,G_J,\mc C, \mc F) = \mc H (\mf s^\vee_\psi, \vec{v})$ have the same
Bernstein presentation. In particular the root data
\begin{equation}\label{eq:4.8}
(X_{\mf f}, R_{\mf f}, X_{\mf f}^\vee, R_{\mf f}^\vee) \quad \text{and} \quad 
\mc R_{\mf s^\vee} = \big( X^* (T_{\mf s^\vee_\psi}), \Phi_{\mf s^\vee_\psi}, 
X_* (T_{\mf s^\vee_\psi}), \Phi_{\mf s^\vee_\psi}^\vee \big)
\end{equation}
are isomorphic. This isomorphism of root data is induced by the $W_{\mf s_\psi}$-equivariant 
bijection $\mf s_{L,\psi} \to \mf s^\vee_{L,\psi}$ from Theorem \ref{thm:4.1}. (The choices 
of basepoints are not needed for this, since an adjustment of a basepoint only multiplies a 
(co)character by a complex number, and that still allows us to detect the same maps between 
(co)character lattices.) Although Theorem \ref{thm:4.1} is only canonical up to twists by
weakly unramified characters, the isomorphism \eqref{eq:4.8} is entirely canonical, for
weakly unramified twists also just multiply (co)characters by nonzero scalars. 
(Such weakly unramified twists may move things to another Bernstein component, but the
Hecke algebra for the new one is canonically identified with the original Hecke algebra.)

Theorem \ref{thm:3.1} entails that there is a unique algebra algebra isomorphism 
\begin{equation}\label{eq:4.12}
\mc H (L_{\mf f}, \hat P_{L,\mf f}, (\hat \sigma \otimes \psi)|_{\hat P_{L,\mf f}} ) 
\cong \mc O (\mf s_{L,\psi})
\end{equation}
such that Theorem \ref{thm:3.4}.b for $\Irr (L_{\mf f})_{\mf s_{L,\psi}}$ just sends
$\pi_L$ to the character of $\mc O (\mf s_{L,\psi})$ given by evaluation at
$(L,\pi_L)$. By Proposition \ref{prop:3.2}, Theorem \ref{thm:3.1}.b, \eqref{eq:3.16}
and \eqref{eq:4.12} the multiplication map
\[
\mc O (\mf s_{L,\psi}) \otimes \mc H (W_{\mf s_\psi}, q_K^{\mc N}) \to
\mc H (G, \hat P_{\mf f}, \hat \sigma \otimes \psi) 
\]
is an isomorphism of vector spaces. Similarly, the Bernstein presentation following
\eqref{eq:2.4} entails that the multiplication map
\[
\mc O (\mf s^\vee_{L,\psi}) \otimes \mc H (W_{\mf s^\vee_\psi}, \vec{v}^{\, 2 \lambda} ) 
\to \mc H (\mf s_\psi^\vee, \vec{v}) 
\]
is a linear bijection. 
Theorem \ref{thm:4.1} also induces an algebra isomorphism
\begin{equation}\label{eq:4.9}
\mc O (\mf s^\vee_{L,\psi}) \to \mc O (\mf s_{L,\psi}) ,
\end{equation}
and Proposition \ref{prop:4.2}.b gives rise to a linear bijection
\begin{equation}\label{eq:4.10}
\mc H (W_{\mf s_\psi}, q_K^{\mc N}) \to \mc H (W_{\mf s^\vee_\psi}, \vec{v}^{\, 2 \lambda} ) ,
\end{equation}
which sends a basis element $T_w$ to a basis element $T_{w^\vee}$.
The maps \eqref{eq:4.9} and \eqref{eq:4.10} will combine to an isomorphism between affine
Hecke algebras, once we make the remaining choices appropriately.

We note that by \eqref{eq:4.8} and Proposition \ref{prop:4.2}.b 
\begin{equation}\label{eq:4.11}
W_{\mf s^\vee_\psi} \cong W_{\mf s_\psi} \cong W(R_{\mf f}) \cong W(\Phi_{\mf s_\psi^\vee}) .
\end{equation}
The set of simple roots $\Delta_{\mf f}$ of $R_{\mf f}$ determines a (unique) basis 
$\Delta_{\mf s_\psi^\vee}$ of $\Phi_{\mf s_\psi^\vee}$ such that \eqref{eq:4.8} becomes
an isomorphism of based root data. As $W_{\mf s^\vee_\psi} = W(\Phi_{\mf s_\psi^\vee})$,
we still had complete freedom to choose a basis for $\mc R_{\mf s^\vee}$ in \eqref{eq:2.4}.

Since $\mc G$ is simple, so is $G^\vee_{\phi_L}$, and $\Phi_{\mf s_\psi^\vee}$ is 
irreducible \cite[\S 2.13]{LusCusp}. Hence the array of parameters $\vec{v}$ reduces to
a single complex number $v$, and we may write $\mc H (\mf s_\psi^\vee,v)$ for
$\mc H (\mf s_\psi^\vee,\vec{v})$. Lusztig showed that for $v = |k|^{1/2} = q_K^{1/2}$
the isomorphisms \eqref{eq:4.8}, \eqref{eq:4.9} and \eqref{eq:4.10} combine to an 
algebra isomorphism
\begin{equation}
\mc H (G, \hat P_{\mf f}, \hat \sigma \otimes \psi) \cong \mc H (\mf s_\psi^\vee, q_K^{1/2}) .
\end{equation}
Notice that on the left hand side we have the parameters $q_K^{\mc N (\alpha)}$ for
$\mc H (G, \hat P_{\mf f}, \hat \sigma \otimes \psi)$ in the Iwahori--Matsumoto presentation,
as given in \cite[\S 1.18]{LusUni1}, whereas on the right hand side we have the parameters
$q_K^{\lambda (\alpha)/2}, q_K^{\lambda^* (\alpha) / 2}$ for $\mc H (\mf s_\psi^\vee, q_K^{1/2})$
in the Bernstein presentation. Transforming one presentation into the other, as in
\cite[\S 5.12]{LusUni1}, yields the required relations between the parameters on both sides.
\end{proof}

\emph{Proof of Theorem \ref{thm:4.5}}\\
The bijection $\mf s_{L,\psi} \to \mf s_{L,\psi}^\vee$ from Theorem \ref{thm:4.1} gives 
an algebra isomorphism $\mc O (\mf s^\vee_{L,\psi}) \to \mc O (\mf s_{L,\psi})$.
From \eqref{eq:4.6} and \eqref{eq:4.7} we know that, when passing to the adjoint group
$\mc G_\ad$, the presentations of the affine Hecke algebras only change in the tori and the
lattices. Therefore we may assume that $\mc G$ is adjoint.

Such a $\mc G$ is a direct product of simple, adjoint $K$-groups, and all objects under
consideration factor accordingly. Thus, we may even assume that $\mc G$ is a simple adjoint
$K$-group. 

Then it is the restriction of scalars of an absolutely (i.e. over $\overline K$) simple
$K'$-group $\mc G'$, for a finite unramified extension $K' / K$. On the $p$-adic side the 
identification $\mc G (K) = \mc G' (K')$ does not change the Hecke algebra of the type.
We should, however, note that the parameter function $q_K^{\mc N}$ for
$\mc H_\af (G, P_{\mf f}, \sigma)$ is now computed as $q_{K'}^{\mc N'}$, where $q_{K'}$
is the cardinality of the residue field of $K'$.

On the Galois side Lemma \ref{lem:2.2} says that $\mc H (\mf s_\psi^\vee, \vec{v})$
is invariant under the Weil restriction $\mc G = \mr{Res}_{K' / K} \mc G'$. Thus we
reduced Theorem \ref{thm:4.5} to Lemma \ref{lem:4.6}. $\qquad \Box$

\section{A local Langlands correspondence}
\label{sec:LLC}

Recall that $\mc G$ is a connected reductive $K$-group, which splits over an unramified 
extension of $K$. As in Section \ref{sec:2}, we consider $\mc G$ as an inner twist 
of a quasi-split $K$-group. Let $\mf s_\psi$ be a unipotent inertial equivalence
class for $\Irr (G)$, as in Section \ref{sec:3}. It is associated to a parahoric
subgroup $P_{\mf f}$, a cuspidal unipotent representation $\sigma$ of $P_{\mf f}$
and an extension $\hat \sigma \otimes \psi$ of $\sigma$ to $\hat P_{\mf f}$.
Moreover $\mf s_\psi$ comes from a cuspidal inertial class $\mf s_{L,\psi}$ for 
a Levi subgroup $L = L_{\mf f}$ of $G = \mc G (K)$. By Proposition \ref{prop:4.2}.a
$\mf s_\psi$ gives rise to an inertial equivalence class $\mf s_\psi^\vee$ of 
enhanced Langlands parameters for $G$.

Theorems \ref{thm:3.4}.b, \ref{thm:4.5} and \ref{thm:2.1} yield bijections
\begin{equation}\label{eq:5.1}
\Irr (G)_{\mf s_\psi} \longrightarrow \Irr \big( \mc H (G, \hat P_{\mf f}, \hat \sigma
\otimes \psi) \big) \longrightarrow \Irr \big( \mc H (\mf s_\psi^\vee, \vec{v}) \big) 
\longrightarrow \Phi_e (G)^{\mf s_\psi^\vee} .
\end{equation}
We note that the third map is canonical and that the second map is canonical
up to certain twists by weakly ramified characters. However, it is unclear how canonical
the first map in \eqref{eq:5.1} is, for $\Rep (G)_{\mf s_\psi}$ may admit several
different types. We will see later that, if we forget the enhancements of the Langlands
parameters at the right hand side of \eqref{eq:5.1}, the map becomes canonical up
to twists by weakly ramified characters of $G$.

\begin{thm}\label{thm:5.1}
The maps \eqref{eq:5.1} combine to a bijection
\[
\begin{array}{ccc}
\Irr (G)_\unip & \longrightarrow & \Phi_{\nr,e}(G) \\
\pi & \mapsto & (\phi_\pi, \rho_\pi)
\end{array}.
\]
\end{thm}
\begin{proof}
Recall from \cite{BeDe} and \eqref{eq:2.10} that 
\[
\Irr (G)_\unip = \bigsqcup_{\mf s \in \mf{Be}(G)_\unip} \Irr (G)_{\mf s} 
\quad \text{ and } \quad \Phi_{\nr,e}(G) = 
\bigsqcup_{\mf s^\vee \in \mf{Be}^\vee (G)_\nr} \Phi_e (G)^{\mf s^\vee} .
\]
In Proposition \ref{prop:4.2}.a we found a bijection $\mf{Be}(G)_\unip \leftrightarrow
\mf{Be}^\vee (G)_\nr$. Combine this with \eqref{eq:5.1}.
\end{proof}

We check that the bijection in Theorem \ref{thm:5.1} satisfies many properties which 
are expected for a local Langlands correspondence.

\begin{lem}\label{lem:5.6}
Theorem \ref{thm:5.1} is compatible with direct products of reductive $K$-groups.
\end{lem}
\begin{proof}
Suppose that $\mc G = \mc G_1 \times \mc G_2$ as $K$-groups. Then all involved objects 
for $\mc G$ factorize naturally are direct products of the analogous objects for 
$\mc G_1$ and $\mc G_2$, for example 
\begin{equation*}
\Phi (G_1 \times G_2) = H^1 (\mb W_K, G_1^\vee \times G_2^\vee) \cong
H^1 (\mb W_K, G_1^\vee) \times H^1 (\mb W_K, G_2^\vee) = \Phi (G_1) \times \Phi (G_2).
\end{equation*}
Our constructions preserve these factorizations, that is implicit in all arguments.
In particular $\pi_1 \otimes \pi_2 \in \Irr_\unip (G_1 \times G_2)$ is mapped to 
$(\phi_{\pi_1} \times \phi_{\pi_2}, \rho_{\pi_1} \otimes \rho_{\pi_2}) 
\in \Phi_{\nr,e} (G_1 \times G_2)$.
\end{proof}

\begin{lem}\label{lem:5.11}
The bijection in Theorem \ref{thm:5.1} is equivariant with respect to the natural
actions of $X_\Wr (G)$.
\end{lem}
\begin{proof}
First we reformulate \eqref{eq:5.1} in a $X_\Wr (G)$-stable setting. By \eqref{eq:3.12}
and \eqref{eq:3.16} 
\[
\mc H (G,P_{\mf f},\sigma) = \bigoplus\nolimits_{\psi \in \Irr (\Omega_{\mf f,tor})} 
\mc H (G, \hat P_{\mf f}, \hat \sigma \otimes \psi) .
\]
Hence the bijection from Theorem \ref{thm:5.1} can also be expressed as
\begin{multline}\label{eq:5.14}
\bigsqcup\nolimits_{\psi \in \Irr (\Omega_{\mf f,tor})} \Irr (G)_{\mf s_\psi} =
\Irr (G)_{(P_{\mf f},\sigma)} \to \Irr (\mc H (G,P_{\mf f},\sigma)) \to \\
\Irr \Big(  \bigoplus_{\psi \in \Irr (\Omega_{\mf f,tor})}
\mc H (\mf s_\psi, \vec{v}) \Big) \to \bigsqcup_{\psi \in \Irr (\Omega_{\mf f,tor})} 
\Phi_{nr,e}(G)^{\mf s^\vee_\psi} .
\end{multline}
It is clear from \eqref{eq:3.1} that the first arrow in \eqref{eq:5.14} is
$X_\Wr (G)$-equivariant. The algebra isomorphism
\begin{equation}\label{eq:5.17}
\mc H (G,P_{\mf f},\sigma) \to 
\bigoplus\nolimits_{\psi \in \Irr (\Omega_{\mf f,tor})} \mc H (\mf s_\psi, \vec{v}) 
\end{equation}
underlying the second arrow in \eqref{eq:5.14} consists of two parts. 
Firstly \eqref{eq:4.10} (which comes from the comparison of Weyl groups in Proposition 
\ref{prop:4.2}.b) and secondly the bijection
\begin{equation}\label{eq:5.16}
\bigoplus\nolimits_{\psi \in \Irr (\Omega_{\mf f,tor})} 
\mc O (\mf s_{L,\psi}^\vee) \to \bigoplus\nolimits_{\psi \in \Irr (\Omega_{\mf f,tor})} 
\mc O (\mf s_{L,\psi}) 
\end{equation}
induced by Theorem \ref{thm:4.1}. By Theorem \ref{thm:4.1}.i, \eqref{eq:5.16} is
$X_\Wr (G)$-equivariant, while \eqref{eq:4.10} is not affected by weakly unramified
characters. Therefore \eqref{eq:5.17} and the second arrow in \eqref{eq:5.14} are
$X_\Wr (G)$-equivariant. By Lemma \ref{lem:2.3} the third arrow in \eqref{eq:5.14}
is $X_\Wr (G)$-equivariant.
\end{proof}

\begin{lem}\label{lem:5.2} 
In Theorem \ref{thm:5.1} $\pi$ is supercuspidal if and only if $(\phi_\pi,\rho_\pi)$
is cuspidal. In that setting Theorem \ref{thm:5.1} agrees with the bijection
from \ref{thm:4.1} and \cite{FOS}.
\end{lem}
\begin{proof}
The first part follows directly from Proposition \ref{prop:4.2}.a.

In the cuspidal case the third map in \eqref{eq:5.1} reduces to 
$\Irr (\mc O (\mf s^\vee_\psi)) \leftrightarrow \mf s_\psi^\vee$, see \eqref{eq:2.11}.
By \eqref{eq:4.12} the first map in \eqref{eq:5.1} becomes the canonical bijection
$\Irr (G)_{\mf s_\psi} \leftrightarrow \Irr (\mc O (\mf s_\psi))$ and \eqref{eq:4.9} says 
the second map in \eqref{eq:5.1} is induced by $\mf s_{\psi}^\vee \leftrightarrow 
\mf s_{\psi}$ from Proposition \ref{prop:4.2}.a. Hence \eqref{eq:5.1} and 
Theorem \ref{thm:5.1} boil down to Theorem \ref{thm:4.1}.
\end{proof}

\begin{lem}\label{lem:5.3}
Let $\mf{Lev}(G)$ be a set of representatives of the Levi subgroups of $G$.
The cuspidal support maps and Theorem \ref{thm:5.1} make a commutative diagram
\[
\begin{array}{ccc}
\Irr_\unip (G) & \longrightarrow & \Phi_{\nr,e}(G) \\
\downarrow \mb{Sc} & & \downarrow \mb{Sc} \\
\bigsqcup\limits_{L \in \mf{Lev}(G)} \Irr_{\cusp,\unip}(L) \big/ N_G (L)  & 
\longrightarrow & \bigsqcup\limits_{L \in \mf{Lev}(G)} 
\Phi_{\nr,\cusp}(L) \big/ N_{G^\vee} (L^\vee \rtimes \mb W_K)
\end{array} .
\]
\end{lem}
\begin{proof}
In the bottom line the actions of $N_G (L)$ and $N_{G^\vee} (L^\vee \rtimes \mb W_K)$
factor through the finite group
\[
N_G (L) / L \cong  N_{G^\vee} (L^\vee \rtimes \mb W_K) /L^\vee . 
\]
In the proof of Proposition \ref{prop:4.2}.a we saw that Theorem \ref{thm:4.1} 
is equivariant for the actions of this group. Now Lemma \ref{lem:5.2} says that the 
bottom line of the diagram is well-defined.

Suppose that (up to $G^\vee$-conjugacy) $\mb{Sc}(\phi_\pi,\rho_\pi) = 
(\chi_L \phi_L, \rho_L)$, with $\chi_L \in X_\nr (L)$. By Theorem \ref{thm:2.1} 
$\bar M (\phi,\rho,\vec{v})$ is a constituent of 
$\mr{ind}_{\mc H (\mf s^\vee_L, \vec{v})}^{\mc H (\mf s^\vee, \vec{v})} 
(L, \chi_L \phi_L,\rho_L)$. Here $\mc H (\mf s^\vee_L, \vec{v})$ is embedded in
$\mc H (\mf s^\vee, \vec{v})$ as $\mc O (\mf s_L^\vee)$. By \eqref{eq:4.9} this 
corresponds to the subalgebra 
\[
\mc O (\mf s_L) \cong \mc H (L, \hat P_{\mf f} \cap L, 
(\hat \sigma \otimes \psi) |_{\hat P_{\mf f} \cap L} ) \quad \text{of} \quad
\mc H (G, \hat P_{\mf f}, \hat \sigma \otimes \psi). 
\]
Hence $\mr{Hom}_{\hat P_{\mf f}} (\hat \sigma \otimes \psi ,\pi)$ is a constituent of 
$\mr{ind}_{\mc H (L, \hat P_{\mf f} \cap L, (\hat \sigma \otimes \psi) |_{\hat P_{\mf f} 
\cap L} )}^{\mc H (G, \hat P_{\mf f}, \hat \sigma \otimes \psi)} (\chi)$,
where $\chi \in \Irr_{\cusp,\unip}(L)$ is the image of $(L, \chi_L \phi_L,\rho_L)$
under Theorem \ref{thm:4.1}. By \eqref{eq:3.36} $\pi$ is a constituent of
$I^G_{L N}(\chi)$, for a suitable parabolic subgroup $L N \subset G$ with Levi factor 
$N$. As $\chi$ is supercuspidal, this means that $(L,\chi)$ is the cuspidal support
of $\pi$ (up to $G$-conjugacy, as always).
\end{proof}

\begin{lem}\label{lem:5.4}
In Theorem \ref{thm:5.1} $\pi \in \Irr_\unip (G)$ is tempered if and only if 
$\phi_\pi \in \Phi_\nr (G)$ is bounded.
\end{lem}
\begin{proof}
Recall that Theorem \ref{thm:5.1} was built from \eqref{eq:5.1}. By Theorem 
\ref{thm:3.4}.c the first map of \eqref{eq:5.1} (and its inverse), preserve 
temperedness. By Theorem \ref{thm:4.5} the same holds for the second map in 
\eqref{eq:5.1}. By Theorem \ref{thm:2.1} the third map in \eqref{eq:5.1} turns 
tempered representations into bounded (enhanced) L-parameters and conversely.
\end{proof}

\begin{lem}\label{lem:5.5}
In Theorem \ref{thm:5.1} $\pi \in \Irr_\unip (G)$ is essentially square-integrable
if and only if $\phi_\pi \in \Phi_\nr (G)$ is discrete.
\end{lem}
\begin{proof}
Consider the chain of maps \eqref{eq:5.1}. By Theorem \ref{thm:3.4}.d the first
of those maps sends essentially square-integrable representations to essentially
discrete series modules. The second map preserves the essentially discrete series
property, becomes it comes from an isomorphism between all the structure defining
these affine Hecke algebras (Theorem \ref{thm:4.5}). 

For unramified enhanced L-parameters $(\phi_L,\rho_L)$, 
\cite[Proposition 3.9.b]{AMS3} says that
\[
R \big( G^\circ_{\phi_1}, Z(L_c^\vee)^{\mb W_K,\circ} \big)_{\mr{red}} = 
R \big( Z_{G_\Sc^\vee}(\phi_L |_{\mb I_K}), Z(L_c^\vee)^{\mb W_K,\circ} \big)_{\mr{red}} 
= R \big( G^\vee_\Sc, Z(L_c^\vee)^{\mb W_K,\circ} \big)_{\mr{red}} .
\]
As $G^\vee_\Sc$ is semisimple, the rank of this root system is 
$\dim_\C (Z(L_c^\vee)^{\mb W_K,\circ})$, which by \cite[Lemma 3.7]{AMS3} equals
\[
\dim_\C \big( Z(L^\vee)^{\mb W_K,\circ} / X_\nr (G) \big) = 
\dim_\C ( T_{\mf s^\vee} / X_\nr (G)) .
\]
Now we can apply Theorem \ref{thm:2.1}.e, which says that the first map in 
\eqref{eq:5.1} sends essentially discrete series modules to discrete enhanced
L-parameters.

These arguments also work for the inverses of the maps in \eqref{eq:5.1}.
\end{proof}

Recall from \cite[p. 20--23]{Lan} and \cite[\S 10.1]{Bor} that every $\phi (G)$
determines a character $\chi_\phi$ of $Z(G)$. For the construction, one first
embeds $\mc G$ in a connected reductive $K$-group $\overline{\mc G}$ with
$\mc G_\der = \overline{\mc G}_\der$, such that $Z(\overline{\mc G})$
is connected. Then one lifts $\phi$ to a L-parameter $\overline \phi$ for
$\overline G = \overline{\mc G}(K)$. The natural projection ${}^L \overline G \to
{}^L Z(\overline G)$ produces an L-parameter $\overline{\phi}_z$ for 
$Z(\overline G) = Z(\overline{\mc G})(K)$, and via the local Langlands correspondence
for tori $\overline{\phi}_z$ determines a character $\chi_{\overline \phi}$ of 
$Z (\overline G)$. Then $\chi_\phi$ is given by restricting $\chi_{\overline \phi}$
to $Z(G)$. Langlands \cite[p. 23]{Lan} checked that $\chi_\phi$ does not depend on
the choices made above.

\begin{lem}\label{lem:5.12}
In Theorem \ref{thm:5.1} the central character of $\pi$ equals the character 
$\chi_{\phi_\pi}$ of $Z(G)$ determined by $\phi_\pi$. This character is
unramified, that is, trivial on every compact subgroup of $Z(G)$.
\end{lem}
\begin{proof}
The construction in \cite[p. 20--21]{Lan} can be executed such
that $\overline{\mc G}$, like $\mc G$, splits over an unramified extension of $K$.
As $\phi_\pi$ is unramified, it can be lifted to a $\overline \phi \in \Phi_\nr (\overline G)$.
Then $\overline{\phi}_z \in \Phi (Z(\overline G))$ is also unramified.
Since $Z(\overline{\mc G})$ is connected and splits over an unramified extension of $K$:
\[
\Phi_\nr (Z(\overline G)) = \big( (Z(\overline{\mc G})^\vee)^{\mb I_K} \big)_{\Fr} =
\big( Z(\overline{\mc G})^\vee \big)_{\Fr} = \big( Z(\overline{\mc G})^\vee \big)^\circ_{\Fr} 
= X_\nr (Z(\overline G)) .
\]
Hence $\chi_{\overline \phi} \in \Irr (Z(\overline G))$ is unramified. Then its restriction
$\chi_\phi$ factors through 
\begin{equation}\label{eq:5.15}
Z(G) / Z(G)_\cpt \cong X_* (Z(\mc G)^\circ (K)). 
\end{equation}
The cuspidal unipotent representation $\bar \sigma$ of $\overline{P_{\mf f}}$ is trivial
on the centre of $\overline{P_{\mf f}}$. Every compact subgroup of $Z(G)$ is contained
in the parahoric subgroup $P_{\mf f}$, so projects to a central subgroup of 
$\overline{P_{\mf f}}$. Hence the kernel of $\sigma \in \Irr (P_{\mf f})$ contains
$Z(G)_\cpt$. As $\pi \in \Rep (G)_{(P_{\mf f},\sigma)}$, this implies 
$Z(G)_\cpt \subset \ker (\pi)$. In other words, the central character $\chi_\pi$ of $\pi$
factors also through \eqref{eq:5.15}.

The lattice of $K$-rational cocharacters of $Z(\mc G)^\circ (K)$ can be identified with
the cocharacter lattice of the maximal $K$-split subtorus $Z(\mc G)_s$ of $Z(\mc G)^\circ$.
Thus $\chi_\pi$ and $\chi_\phi$ are determined by their restrictions to 
$Z (G)_s = Z(\mc G)_s (K)$, which both are unramified characters.

This means that, to prove the lemma, it suffices to compare the characters of $Z(G)_s$
determined by $\pi$ and by $\lambda_\pi$. The latter admits a more direct description than
$\chi_\phi$. Namely, the inclusion $Z(\mc G)_s \to \mc G$ has a dual surjection
$\mc G^\vee \to (Z(\mc G)_s)^\vee$. The image $\phi_s \in \Phi (Z(G)_s)$ of 
$\phi_\pi$ determines a character $\chi_{\phi_s}$, which equals the restriction of 
$\chi_{\phi_\pi}$ of $Z(G)_s$.

Now we reduce to the cuspidal case. It is clear that $\pi$ and $\mb{Sc}(\pi)$ have the same
$Z(G)$-character. Let us write 
\[
\mb{Sc}(\phi_\pi,\rho_\pi) = (L,\phi_L,\rho_L) \quad \text{with} \quad 
(\phi_L,\rho_L) \in \Phi_\cusp (L).
\]
From \cite[Lemma 7.6 and Definition 7.7]{AMS1} we see that $\phi_L |_{\mb I_K} =
\phi_\pi |_{\mb I_K}$ and $\phi_L (\Fr) = \phi_\pi (\Fr) t$, where $t \in G^\vee_\der$
commutes with the image of $\phi_\pi$. As 
\[
G^\vee_\der \subset \ker \big( G^\vee \to Z(\mc G)_s^\vee (\C) \big) ,
\]
$\phi_L$ and $\phi_\pi$ have the same image $\phi_s$ in $\Phi (Z(G)_s)$. So if we replace
$(\phi_\pi,\rho_\pi)$ by its cuspidal support $(L,\phi_L,\rho_L)$, we do not change the
$Z(G)_s$-character $\chi_{\phi_s}$. Although $\mb{Sc}(\phi_\pi,\rho_\pi)$ is determined 
only up to $G^\vee$-conjugacy, we may pick any representative for it, because conjugation
by elements of $G^\vee$ does not affect $\phi_s$. In view of this and Lemma \ref{lem:5.3}
we may assume that $\pi$ is supercuspidal. 

Now, as explained after (15.5) in \cite{FOS}, $\pi$ can be written as $\pi' \otimes \chi$
with $\pi' \in \Irr_\unip (G / Z(G)_s)$ and $\chi \in X_\nr (G)$. Clearly
$\chi_\pi |_{Z(G)_s} = \chi |_{Z(G)_s}$. The construction in \cite[(15.6) and (15.10)]{FOS}
says that $(\phi_\pi,\rho_\pi) = (\phi_{\pi'} \hat \chi , \rho_\pi)$, where $\hat \chi \in
Z(\mc G^\vee)^\circ_\Fr$ is the image of $\chi$. We see that $\phi_s$ equals the L-parameter
of $\chi |_{Z(G)_s}$ and hence
\[
\chi_\phi |_{Z(G)_s} = \chi_{\phi_s} = \chi |_{Z(G)_s} = \chi_\pi |_{Z(G)_s} .
\]
As discussed above, this implies that $\chi_\phi = \chi_\pi$ on $Z(G)$.
\end{proof}

Let $\mc Q \mc U_Q$ be a parabolic $K$-subgroup of $\mc G$, with unipotent radical $\mc U_Q$
and Levi factor $\mc Q$. Suppose that $\phi \in \Phi (G)$ factors via ${}^L Q$. Then we
can compare representations of $G$ and $Q$ associated to enhancements of $\phi$, via 
normalized parabolic induction. Let $p_\zeta$ and $p_{\zeta^Q}$ be as in \eqref{eq:2.13}. 
By \cite[Theorem 7.10.a]{AMS1} there is a natural injection
\[
p_{\zeta^Q} \C [\mc S_\phi^Q] \to p_\zeta \C [\mc S_\phi] .
\]
This enables us to retract $G$-relevant enhancements of $\phi$ to representations of $\mc S_\phi^Q$.

\begin{lem}\label{lem:5.15}
Let $\phi \in \Phi_\nr (Q)$.
\enuma{
\item Suppose that the function $\epsilon_{u_\phi ,j}(\phi (\Fr_K) \phi_b (\Fr_K)^{-1} ,\vec{v})$
from \eqref{eq:2.14} and \eqref{eq:2.15} is nonzero. Let $\rho \in \Irr (\mc S_\phi)$ be $G$-relevant 
and let $\rho^Q \in \Irr (\mc S_\phi^Q)$ be $Q$-relevant. 

Then the multiplicity of $\pi (\phi,\rho) \in \Irr_\unip (G)$ as a constituent of 
$I_{Q U_Q}^G \pi (\phi,\rho^Q)$ is $[ \rho^Q : \rho ]_{\mc S_\phi^Q}$. 
It already appears that many times as a quotient. 
\item Let $(\phi,\rho^Q) \in \Phi_{\nr,e}(Q)$ be bounded. Then
\[
I_{Q U_Q}^G \pi (\phi,\rho^Q) \cong 
\bigoplus\nolimits_\rho \mr{Hom}_{\mc S_\phi^Q} (\rho^Q ,\rho) \otimes \pi (\phi,\rho) ,
\]
where the sum runs over all $\rho \in \Irr (\mc S_\phi)$ with 
$\mb{Sc}(\phi,\rho) = \mb{Sc}(\phi,\rho^Q )$.
}
\end{lem}
\begin{proof}
(a) By \cite[Lemma 3.16.b]{AMS3} this holds for the modules $\bar M (\phi,\rho,\vec{v})$ and
$\mr{ind}_{\mc H (\mf s^\vee_{Q,\psi}, \vec{v}) }^{\mc H (\mf s^\vee_\psi,\vec{v})} 
\bar M (\phi,\rho^Q,\vec{v})$. Theorem \ref{thm:4.5} transfers that to a statement about modules
for $\mc H (G,\hat P_{\mf f}, \hat \sigma \otimes \psi)$. By Theorem \ref{thm:3.4}.b and \eqref{eq:3.36}
it becomes the desired statement about elements of Rep$(G)_{\mf s_\psi}$.\\
(b) By Lemma \ref{lem:2.3}.c $\bar M (\phi,\rho,\vec{v})$ and $\bar{M}^Q (\phi,\rho^Q,\vec{v})$ are
equal to the standard modules with the same parameters. Knowing that, \cite[Lemma 3.16.a]{AMS3}
gives the desired statement for $\mc H (\mf s^\vee_\psi,\vec{v})$-modules. 
As in the proof of part (a), that can be transferred to elements of Rep$(G)_{\mf s_\psi}$.
\end{proof}

Finally we work out the compatibility of Theorem \ref{thm:5.1} with the Langlands classification
for representations of reductive $p$-adic groups \cite{Kon,Ren}. We briefly recall the statement.

For every $\pi \in \Irr (G)$ there exists a triple $(P,\tau,\nu)$, unique up to $G$-conjugation, 
such that: 
\begin{itemize}
\item $P$ is a parabolic subgroup of $G$;
\item $\tau \in \Irr (P/U_P)$ is tempered, where $U_P$ denotes the unipotent radical of $P$;
\item the unramified character $\nu : P/U_P \to \R_{>0}$ is strictly positive
with respect to $P$;
\item $\pi$ is the unique irreducible quotient of $I_P^G (\tau \otimes \nu)$.
\end{itemize}
The Langlands classification for (enhanced) L-parameters \cite{SiZi} was already discussed before
Proposition  \ref{prop:2.4} -- we use the notations from over there.

\begin{lem}\label{lem:5.16}
Let $(\phi,\rho) \in \Phi_{\nr,e}(G)$ and let $(Q U_Q,\phi_b,z)$ be the triple associated to
$\phi$ by \cite[Theorem 4.6]{SiZi}. Recall from \eqref{eq:2.13} that $\rho$ can also be considered
as enhancement of $\phi$ or $\phi_b$ as L-parameters for $Q$.
\enuma{
\item $\pi (\phi,\rho)$ is the unique irreducible quotient of $I_{Q U_Q}^G \pi^Q (\phi,\rho)$.
\item $\pi^Q (\phi,\rho) = z \otimes \pi^Q (\phi_b, \rho)$ with $\pi^Q (\phi_b, \rho) \in \Irr_\unip (Q)$
tempered and $z \in X_\nr (Q)$ strictly positive with respect to $Q U_Q$.
The data for $\pi (\phi,\rho)$ in the Langlands classification for Rep$(G)$ are
$(Q U_Q, \pi^Q (\phi_b ,\rho), z)$, up to $G$-conjugacy.
}
\end{lem}
\begin{proof}
(a) By Lemma \ref{lem:4.6} (and the proof of Theorem \ref{thm:4.5}) $\vec{v} \in \R_{>1}^d$. Thus
we may apply Proposition \ref{prop:2.4}, which says that the analogous statement for 
$\mc H (\mf s^\vee, \vec{v})$-modules holds. With Theorem \ref{thm:4.5}, Theorem \ref{thm:3.4}.b
and \eqref{eq:3.36} we transfer that to Rep$(G)_{\mf s_\psi}$.\\
(b) The first part follows from Lemmas \ref{lem:5.11} and \ref{lem:5.4}. The second part is a 
consequence of the uniqueness (up to conjugacy) of the Langlands data of $\pi (\phi,\rho)$.
\end{proof}

\end{document}